\documentclass[12pt]{article}
\usepackage{amsmath,amssymb,amsthm}

 \newcommand{\beqn}{\begin{eqnarray}}
 \newcommand{\eeqn}{\end{eqnarray}}
 \newcommand{\be}{\begin{equation}}
 \newcommand{\ee}{\end{equation}}
 \newcommand{\ba}{\begin{array}}
 \newcommand{\ea}{\end{array}}
 
 \newcommand{\pa}{\partial}
 \newcommand{\re}{\ref}
 \newcommand{\ci}{\cite}
 \newcommand{\ds}{\displaystyle}
 \newcommand{\la}{\label}

\newcommand{\fr}{\frac}

\newcommand{\ov}{\overline}

\newcommand{\ti}{\tilde}
\newcommand{\supp}{{\rm supp~}}

\newcommand{\M}{{\cal M}}
\newcommand{\cO}{{\cal O}}

\newcommand{\ve}{\varepsilon}

\newcommand{\de}{\delta}
\newcommand{\al}{\alpha}

\newcommand{\5}{\hspace{0.5mm}}

\newcommand{\vp}{\varphi}

\newcommand{\si}{\sigma}
\newcommand{\om}{\omega}

\newcommand{\Si}{\Sigma}
\newcommand{\lam}{\lambda}
\newcommand{\Lam}{\Lambda}

\newcommand{\Spec}{{\rm Spec\5}}

\newcommand\C{{\mathbb C}}
\newcommand\R{{\mathbb R}}

\renewcommand{\Re}{\mathop{\mathrm{Re}}}
\renewcommand{\Im}{\mathop{\mathrm{Im}}}

 \renewcommand{\theequation}{\thesection.\arabic{equation}}
\renewcommand{\thesection}{\arabic{section}}
\renewcommand{\thesubsection}{\arabic{section}.\arabic{subsection}}
\newtheorem{theorem}{Theorem}[section]

\renewcommand{\thetheorem}{\arabic{section}.\arabic{theorem}}
\newtheorem{definition}[theorem]{Definition}

\newtheorem{lemma}[theorem]{Lemma}
\newtheorem{remark}[theorem]{Remark}
\newtheorem{remarks}[theorem]{Remarks}
\newtheorem{cor}[theorem]{Corollary}
\newtheorem{pro}[theorem]{Proposition}

\newcommand{\bd}{\begin{definition}}
 \newcommand{\ed}{\end{definition}}
\newcommand{\bt}{\begin{theorem}}
 \newcommand{\et}{\end{theorem}}
\newcommand{\bp}{\begin{pro}}
 \newcommand{\ep}{\end{pro}}

\newcommand{\bl}{\begin{lemma}}
 \newcommand{\el}{\end{lemma}}
\newcommand{\bco}{\begin{cor}}
 \newcommand{\eco}{\end{cor}}
 \newcommand{\bce}{\begin{center}}
 \newcommand{\ece}{\end{center}}

\newcommand{\br}{\begin{remark} }
 \newcommand{\er}{\end{remark}}
\newcommand{\brs}{\begin{remarks}}
 \newcommand{\ers}{\end{remarks}}

\begin{document}
%%%%%%%%%%%%%%%%%%%%%%%%%%%%%
%%%%%%%%%%%%%%%%%%%%%%%%%%%%%%%%%%%%%%%%%%%%%%%%%%%%%%%%%%%%%%%%%%%%%%%%%%%%%%%%%%%%%
\begin{titlepage}
\bigskip\bigskip\bigskip

\begin{center}
{\Large\bf
On asymptotic stability of kink for  \bigskip\\
relativistic Ginzburg-Landau equation}
\vspace{1cm}
\\
{\large E.~A.~Kopylova}
\footnote{Supported partly by the grants of
DFG, FWF and RFBR.}\\
{\it Institute for Information Transmission Problems RAS\\
B.Karetnyi 19, Moscow 101447,GSP-4, Russia}\\
e-mail:~elena.kopylova@univie.ac.at
\medskip\\
{\large A.~I.~Komech}$^{\mbox{\scriptsize 1,}}\!\!$
\footnote{Supported partly by the Alexander von Humboldt
Research Award.}\\
{\it Fakult\"at f\"ur Mathematik, Universit\"at Wien\\
and Institute for Information Transmission Problems RAS}\\
 e-mail:~alexander.komech@univie.ac.at

\end{center}

\date{}
%%%%%%%%%%%%%%%%%%%%%%%%%%%%%%%%%%%%%%%%%%%%%%%%%%%%%%%%%%%%%%%%%%%%%%%%%%%%%%%
\vspace{0.5cm}

\begin{abstract}
\noindent We prove the asymptotic stability of standing kink for the
nonlinear relativistic wave equations of the Ginzburg-Landau type in
one space dimension: for any odd initial condition in a small
neighborhood of the kink,  the solution, asymptotically in time, is
the sum of the kink and dispersive part described by the free
Klein-Gordon equation. The remainder converges to zero in a global
norm. Crucial role in the proofs play our recent results on the
weighted energy decay for the Klein-Gordon equations.

\smallskip

\smallskip

\end{abstract}

\end{titlepage}
%%%%%%%%%%%%%%%%%%%%%%%%%%%%%%%%%%%%%%%%%%%%%%%%%%%%%
%%%%%%%%%%%%%%%%%%%%%%%%%%%%%

%%%%%%%%%%%%%%%%%%%%%%%%%%%%%%%%%%%%%%%%%%%%%%%%%%%%
%%%%%%%%%%%%%%%%%%%%%%%%%%%%%%

\section{Introduction}

%%%%%%%%%%%%%%%%%%%%%%%%%%%%%%%%%%%%%%%%%%%%%%%%%%%%%%
%%%%%%%%%%%%%%%%%%
%%%%%%%%%%%%%%%%%%%%%%%%%%%%%%%%%%%%%%%%%%%%%%%%%%%%%%
%%%%%%%%%%%%%%%%%%%%%%%%%%%%%
%%%%%%%%%%%%%%%%%%%%%%%%%%%%%%%%%%%%%%%%%%%%%%%%%%%%%%
%%%%%%%%%%%%%%%%%%
%%%%%%%%%%%%%%%%%%%%%%%%%%%%%%%%%%%%%%%%%%%%%%%%%%%%%%
%%%%%%%%%%%%%%%%%%%%%%%%%%%%%

We prove the asymptotic stability of kinks for relativistic
nonlinear wave equations with two-well potentials of Ginzburg-Landau
type. The work is inspired by the problem of stability of elementary
particles which are modeled as solitary waves of the equations. We
consider the  equation
\be\la{e} \ddot\psi(x,t)=\psi''(x,t)+
F(\psi(x,t)),\quad x\in\R
\ee
where $\psi(x,t)$ is a real solution,
and $F(\psi)=-U'(\psi)$. We consider the potentials $U(\psi)$
similar to the Ginzburg-Landau potential $U_0(\psi)=(\psi^2-1)^2/4$
which correpsonds to the cubic equation with $F(\psi)=\psi-\psi^3$.
%%%%%%%%%%%%%%%%%%%%%%%%%%%%%%%%%%%%%%%%%%%%%%%%%%%%%%%%%%%%%%%%%%%%%%%%%%%%%%%%
\medskip\\
{\bf Condition U1}. {\it The potential $U(\psi)$ is a real 
 smooth even function
which satisfies the
following conditions with some $a,m>0$
and sufficienly large $k>0$,}
\be\la{U1}
U(\psi)>0~~ {\rm for}~~ \psi\not=\pm a~~~~~~~
\ee
\be\la{U11}
~~~~~~~~~U(\psi)=\fr{m^2}2(\psi\mp a)^2+\cO(|\psi\mp a|^{2k}),~~~~~~~
x\to\pm a
\ee
In the vector form, equation (\ref{e}) reads
\beqn\la{eq}
\left\{
\ba{lll}
 \dot{\psi}(x,t)=\pi (x,t)\\
 ~ &  \\
 \dot{\pi}(x,t)=\psi'' (x,t)+  F(\psi(x,t)),\quad x\in\R
\ea\right.
\eeqn
Formally it is a Hamiltonian system with the Hamilton functional
\be\la{Ham}
 {\cal H}(\psi,\pi)=\int\Big[\fr{|\pi(x)|^2}2+\fr{|\psi'(x)|^2}2
 + U(\psi(x))\Big]~dx
\ee
The corresponding stationary equation reads
\be\la{steq}
 s'' - U'(s)=0
\ee
The constant solutions of the stationary equation are:
$ \psi\equiv\pm a$ - stable stationary solutions, and
$\psi\equiv 0$ -- unstable stationary solution.
There is also a "kink", i.e. an odd nonconstant finite energy solution $s(x)$
to (\re{steq}) such that
\be\la{kink}
s(0)=0,~~~ s(x)\to \pm a ~~{\rm as}~~x\to\pm\infty
\ee
The condition {\bf U1} implies that
$(s(x)\mp a)''\sim m^2 (s(x)\mp a)$ for $x\to\pm\infty$, hence
\be\la{s-decay}
s(x)\mp a\sim Ce^{-m|x|},\quad x\to\pm\infty
\ee
The generator of linearized equation near the kink reads
(see Section \ref{main-sec})
$$
  A=\left(\ba{cc}
       0    &      1 \\

      -H    &      0
  \ea\right)
$$
where $H$ is the Schr\"odinger operator
\be\la{AH}
H=-\ds\frac{d^2}{dx^2}-F'(s)=-\ds\frac{d^2}{dx^2}+m^2+ V(x),\quad
V(x)=-F'(s(x))-m^2=U''(s(x))-m^2
\ee
By (\ref{s-decay}), we have
\be\la{V-decay}
 V(x)\sim C(s(x)\mp a)^{2k-2}\sim Ce^{-(2k-2)m|x|},\quad x\to\pm\infty.
\ee The continuous spectrum of the operator $H$ is $\Spec_c
H=[m^2,\infty)$. By technical reasons, we restrict in this paper to
odd solutions $\psi(-x,t)=-\psi(x,t)$. We assume the following
spectral conditions:
%%%%%%%%%%%%%%%%%%%%%%%%%%%%%%%%%%%%%%%%%%%%%%%%%%%%%%%%%%%%%%%%%%%%%%%%%%%%
\medskip\\
{\bf Condition U2}. {\it
The discrete spectrum of the operator $H$ restricted to the subspace
of odd functions consists  only of one simple eigenvalue
\be\la{la1}
\lam_1<m^2,~~~~~~~~~4\lam_1>m^2
\ee
We also consider the edge point
$\lam=m^2$ of the continuous spectrum and assume that
\be\la{la2}
\mbox{\rm $\lam=m^2$ is not eigenvalue nor resonance for the
Schr\"odinger operator $H$}
\ee }
%%%%%%%%%%%%%%%%%%%%%%%%%%%%%%%%%%%%%%%%%%%%%%%%%%%%%%%%%%%%%%%%%%%%%%%%%%%%
We assume also a non-degeneracy condition ``Fermi Golden Rule"
introduced by Sigal
\ci{Sigal93}. The condition
provides a strong coupling  of the nonlinear term with the
eigenfunctions of the continuous spectrum and the energy radiation.
\medskip\\
%%%%%%%%%%%%%%%%%%%%%%%%%%%%%%%%%%%%%%%%%%%%%%%%%%%%%%%%%%%%%%%%%%%%%%%%%%%%
{\bf Condition U3}.
{\it The non-degeneracy condition holds
(cf. condition (1.0.11) in \cite{BS})
\be\la{FGR}
 \int_0^\infty\vp_{4\lam_1}(x)
F''(s(x))\vp_{\lam_1}^2(x)dx\not=0
\ee
where $\vp_{\lam_1}(x)$ and $\vp_{4\lam_1}(x)$
are the odd eigenfunctions
of discrete and continuous spectrum
corresponding to $\lam_1$ and  $4\lam_1$ respectively.
}
\medskip\\
%%%%%%%%%%%%%%%%%%%%%%%%%%%%%%%%%%%%%%%%%%%%%%%%%%%%%%%%%%%%%%%%%%%%%%%%%%%%%%
The Ginzburg-Landau potential $U_0(\psi)=(\psi^2-1)^2/4$
satisfies all the conditions  {\bf U1--U3}      except
(\re{U11}).
In Appendix C we construct small perturbations of the
Ginzburg-Landau potential which satisfy
all the conditions {\bf U1--U3} including (\re{U11}).

Our main results are the following asymptotics
\be\la{Si}
 (\psi(x,t), \dot\psi(x,t))\sim
 (s(x), 0)+W_0(t){\Phi}_{\pm},\quad t\to\pm\infty
\ee
for solutions to  (\re{eq}) with odd initial data close
to the kink $S(x)=(s(x),0)$.
Here $W_0(t)$ is the dynamical group of the free Klein-Gordon equation,
${\Phi}_\pm$ are the corresponding asymptotic  states,
and the remainder converges to zero
$\sim t^{-1/3}$
in the ``global energy norm''
of the Sobolev space $H^1(\R)\oplus L^2(\R)$.
We consider the odd initial data to fix the limit standing kink:
otherwise, the asymptotic holds with a moving kink that we will
consider elsewhere.

\br
We consider the solutions close to the kink, $\psi(x,t)=s(x)+\phi(x,t)$, with 
small perturbations $\phi(x,t)$. For such solution the condition 
(\re{U11}) and the asymptotics (\re{s-decay})
mean that the equation (\re{e}) is almost linear
for large $|x|$. This fact is helpful for application 
of the dispersive properties of the corresponding linearized equation.

\er
Let us comment on previous results in this field.
\medskip\\
$\bullet$ {\it The Schr\"odinger equation}
The asymptotics of type (\re{Si}) were established
for the first time
by Soffer and Weinstein \ci{SW1,SW2} (see also \ci{PW}) for nonlinear
$U(1)$-invariant Schr\"odinger equation with a potential for
small initial states if the nonlinear coupling constant is
sufficiently small.

The results have been extended by Buslaev and Perelman \ci{BP1}
to the translation invariant
1D nonlinear  $U(1)$-invariant
Schr\"odinger equation.
The initial states are sufficiently close to the  solitary waves with
the unique eigenvalue $\lam=0$ in the dicrete spectrum of the
corresponding linearized dynamics.
The novel techniques
\ci{BP1} are based on the "separation of variables"
along the solitary manifold and in transversal directions.
The symplectic projection allows to exclude from the transversal
dynamics the unstable directions corresponding to the zero
discrete spectrum of the linearized dynamics.
The extensions to higher dimensions were obtained
in \ci{Cu01,KZ07,RSS05,TY02}.

Similar techniques were developed by Miller, Pego and Weinstein for the
1D modified KdV and RLW equations, \ci{MW96,PW94}.
These techniques were motivated by the investigation
of soliton asymptotics for integrable equations
(a survey can be found in \ci{DIZ} and \ci{FT}),
and by the methods introduced in \ci{SW1,SW2,W85}.

The techniques  were developed in \ci{BP2,BS}
for the Schr\"odinger
equations in more complicated spectral situation with presence
of a nonzero eigenvalue in the linearized dynamics.
In that case the transversal dynamics inherits the nonzero
discrete spectrum.
Now the decay for the transversal dynamics is obtained by the
reduction to the Poincar\'e  normal form which makes obvious that the
decay depends on the Fermi Golden Rule condition  \ci{MS,Sigal93}.
The condition states a strong interaction of the nonlinear term with
the eigenfunctions of the continuous spectrum which
provides the dispersive
energy radiation to infinity and the decay for the
transversal dynamics.
The extension to higher dimensions were obtained
in \ci{Cu03,Cu2008,SW04}.
Tsai \ci{Tsai2003} developed the techniques
in presence of
an arbitrary finite number of discrete eigenvalues in the
linearized dynamics.
\medskip\\
$\bullet$
{\it Nonrelativistic Klein-Gordon  equations}
The asymptotics of type
(\re{Si})
 were extended to the nonlinear 3D
Klein-Gordon equations with a potential \ci{SW99}, and
for translation invariant system of the 3D
Klein-Gordon equation coupled to a particle \ci{IKV05}.
\medskip\\
$\bullet$ {\it Wave front of 3D Ginzburg-Landau  equation}
The asymptotic stability of wave front
 were proved for 3D relativistic Ginzburg-Landau equation
with initial data which differ from the wave front  on a compact set
\cite{Cu}. The equation differs from the 1D equation (\ref{e}) by the
additional 2D Laplacian. The additional Laplacian improves the dispersive
decay for
the corresponding linearized  Klein-Gordon equation
in the continuous spectral space
 that provides the needed
decay for the transversal dynamics.
\medskip\\
$\bullet$ {\it Orbital stability of the kinks}
For 1D relativistic nonlinear Ginzburg-Landau equations (\re{e})
the orbital stability of the kinks has been proved
in \ci{HPW}.
\medskip

The proving of the asymptotic stability of the kinks for
relativistic
  equations remained an open problem till
now. Main obstacle was the slow decay $\sim t^{-1/2}$ for the free
1D  Klein-Gordon equation (see the discussion in
\ci[Introduction]{Cu}).

Let us comment on our approach.
We  follow general strategy of
\ci{BP1,BP2,BS,Cu,Cu01,Cu03,Cu2008,IKV05,SW99,TY02,Tsai2003}:
linearization of the transversal equations and
further Taylor expansion of the nonlinearity, the Poincar\'e normal
forms and Fermie Golden Rule, etc. We develop for relativistic
equations general scheme which is common in almost all papers in
this area: dispersive and $L^1-L^\infty$ estimates for the linearized equation, virial  estimates
for the nonlinear equation, and method of
majorants. However, the corresponding statements and their proofs in
the context of relativistic equations are completely new.
 Let us comment on our
novel techniques.
\medskip\\
i) The slow decay $\sim t^{-1/2}$ for the free 1D  Klein-Gordon equation
corresponds to the presence of the resonances at the
ends of the continuous
spectrum. We overcome the difficulty
developing our recent
result \cite{1dkg}  {\it identifying} the
slow decaying component
with the contribution of the resonances
of the free 1D  Klein-Gordon equation.
More precisely,
we prove  that
the contribution of the high energy spectrum decays like
$\sim t^{-3/2}$,
in the weighted energy norms. This result plays the
crucial role in our paper,
and provides the decay $\sim t^{-3/2}$ for the
transversal linearized dynamics since the end points of
continuous spectrum
are not resonances in our case due to the antisymmetry
of the solutions.
\medskip\\
ii) The "virial type" estimate (\re{solt}) for the nonlinear
Ginzburg-Landau equation (\re{e}) is novel relativistic version  of
the bound \ci[(1.2.5)]{BS} for the nonlinear Schr\"odinger
equations;
\medskip\\
iii) We give the complete proof of the dispersive estimate
(\re{t-as2});
\medskip\\
iv) We establish an  appropriate relativistic version (\ref{t-as3})
of $L^1\to L^\infty$ estimates;
\medskip\\
v) We prove novel optimal decay estimate (\re{dtQm})  for the
dynamical group of the free 1D Klein-Gordon  equation;
\medskip\\
vi) We give the complete proof of the soliton asymptotics (\re{Si}).
In the context of the Schr\"odinger equation,
the proof of the corresponding asymptotics were sketched
in
\ci{BS}.
\medskip\\
vii) Finally, we construct the examples of the potentials satisfying
all our spectral conditions including the Fermie Golden Rule. The
examples were never constructed in all previous papers in this area.
\medskip\\

Our paper is organized as follows. In Section \ref{main-sec} we
formulate the  main theorem. The linearization at the kink is
carried out in Section \ref{lin-sec}.
In Section \ref{modsec} we derive the dynamical equations
for  the ``discrete" and ``continuous" components of the solution.
In Section \ref{eqns-trans} we transform the dynamical equations
to a Poincare ``normal form". We apply the method of majorants
in Section \ref{maj}.
Finally, in Section \ref{solas-sec} we obtain the soliton
asymptotics (\ref{Si}).

In Appendices A and B we prove the key estimates
(\re{solt}) and (\re{t-as2}).
In Appendix C we construct the examples of the potentials.

{\bf Acknowledgements}
The authors thank V.S.~Buslaev,  H.~Spohn, and
 M.I.~Vishik for fruitful discussions.

%%%%%%%%%%%%%%%%%%%%%%%%%%%%%%%%%%%%%%%%%%%%%%%%%%%%%%%%%%%%%%%%%%%%%%%%%%%%%
%%%%%%%%%%%%%%%%%%%%%%%%%%%%%%%%%%%%%%%%%%%%%%%%%%%%%%%%%%%%%%%%%%%%%%%%%%%%%

\section{Main results}\la{main-sec}
%%%%%%%%%%%%%%%%%%%%%%%%%%%%%%%%%%%%%%%%%%%%%%%%%%%%%%%%%%%%%%%%%%%%%%%%%%%%%%
%%%%%%%%%%%%%%%%%%%%%%%%%%%%%%%%%%%%%%%%%%%%%%%%%%%%%%%%%%%%%%%%%%%%%%%%%%%%%%%
We consider the Cauchy problem for the Hamilton system (\re{Eq})
which we write as
\be\la{Eq}
\dot Y(t)={\cal F}(Y(t)),\quad t\in\R:\quad Y(0)=Y_0.
\ee
Here $Y(t)=(\psi(t), \pi(t))$, $Y_0=(\psi_0, \pi_0)$, and all
derivatives are understood in the sense of distributions.
To formulate our results precisely, let us
first we introduce a suitable phase space for the Cauchy problem
(\ref{Eq}).
We will consider only
odd states  $Y=(\psi,\pi)$:
\be\la{odd}
\psi(-x)=-\psi(x),~~~~~\pi(-x)=-\pi(x),~~~~~~x\in\R
\ee
The space of the odd states is invariant with respect to dynamical
equation (\re{e}) since the potential $U(\psi)$ is even function
according to (U), and hence $F(\psi)$ is the odd function.

For
 $\si\in\R$, and  $l=0,1,2,...$, $p\ge 1$, let us denote by  $W^{p,l}_{\si}$,
the  weighted Sobolev space of the odd
functions with the finite norm
$$
\Vert\psi\Vert_{W^{p,l}_\si}=
\sum\limits_{k=0}^l\Vert(1+|x|)^{\si}\psi^{(k)}\Vert_{L^{p}}<\infty
$$
and  $H^l_\si:=W^{2,l}_\si$, so $H^0_\si=L^2_\si$.
%%%%%%%%%%%%%%%%%%%%%%%%%%%%%%%%%%%%%%%%%%%%%%%%%%%%%%%%%%%%%%%%%%%%%%%%%%%%%%
\begin{definition}\la{def-E}
i) $E_{\si}:=H^1_{\si}\oplus L^2_{\si}$ is the space  of the odd
states $Y=(\psi,\pi)$ with finite norm
\be\la{nEal}
\Vert \,Y\Vert_{E_{\si}}=
\Vert \psi \Vert_{H^1_\si} +\Vert\pi \Vert_{L^2_\si}
\ee
ii) The phase space  ${\cal E}:=S+E$, where  $E=E_0$ and
$S=(s(x),0)$. The metric in ${\cal E}$ is defined as
\be\la{enm}
\rho_{\cal E} (Y_1, Y_2)=\Vert  Y_1-Y_2 \Vert_{E},\quad Y_1,Y_2\in{\cal E}
\ee
iii) $W:=W^{1,2}_0\oplus W^{1,1}_0$ is the space of the odd
states $Y=(\psi,\pi) $ with finite norm
\be\la{nW}
\Vert \,Y\Vert_{W}=
\Vert \psi \Vert_{W^{1,2}_0} +\Vert\pi \Vert_{W^{1,1}_0}
\ee
\end{definition}
%%%%%%%%%%%%%%%%%%%%%%%%%%%%%%%%%%%%%%%%%%%%%%%%%%%%%%%%%%%%%%%%%%%%%%%%%%%%%%%%%%%
Obviously, the Hamilton functional (\re{Ham})
is continuous on the phase space ${\cal E}$.
The existence and uniqueness of the solutions to the Cauchy problem
(\ref{Eq}) follows by methods  \cite {Li,Re,St78}:
\begin{pro}
i) For any initial data $Y_0\in{\cal E}$ there exists the unique solution
$Y(t)\in C(\R,\cal E)$ to the problem (\ref{Eq}).\\
(ii) For every $t\in\R$, the map $U(t): Y_0\mapsto Y(t)$ is continuous in
${\cal E}$.\\
(iii)
The  energy is conserved, i.e.
\be\la{E}
{\cal H}(Y(t))= {\cal H}(Y_0),\,\,\,\,\,t\in\R
\ee
\end{pro}
%%%%%%%%%%%%%%%%%%%%%%%%%%%%%%%%%%%%%%%%%%%%%%%%%%%%%%%%%%%%%%%%%%%%%%%%%%%%%%%%%%
The main result of our paper is the following theorem
\begin{theorem}\la{main}
Let  the potential $U$ satisfy the conditions (U1)-(U3) with $k=7$, 
and let $Y(t)$
be the solution to the Cauchy problem  (\ref{Eq}) with any
initial state $Y_0\in{\cal E}$ which is sufficiently close to the kink:
\be\la{close}
  Y_0=S+X_0,\,\,\,\,d_0:=\Vert X_0\Vert_{E_\si\cap W}\ll 1
\ee
where $\si >5/2$.
Then the asymptotics hold
\be\la{S}
Y(x,t)=(s(x),0)+W_0(t)\Phi_\pm+r_\pm(x,t),\quad t\to\pm\infty
\ee
where $\Phi_\pm\in E$, and $W_0(t)=e^{A_0t}$ is the dynamical group of the
free Klein-Gordon equation (see (\ref{A0}), while
\be\la{rm}
\Vert r_\pm(t)\Vert_{E}=\cO(|t|^{-1/3})
\ee
\end{theorem}
It suffices to  prove the asymptotics  (\re{S}) for $t\to+\infty$
since the system (\re{eq}) is time reversible.
%%%%%%%%%%%%%%%%%%%%%%%%%%%%%%%%%%%%%%%%%%%%%%%%%%%%%%%%%%%%%%%%%%%%%%%%%%%%%%%%
%%%%%%%%%%%%%%%%%%%%%%%%%%%%%%%%%%%%%%%%%%%%%%%%%%%%%%%%%%%%%%%%%%%%%%%%%%%%%%%%%

\section{Linearization at the kink}\la{lin-sec}
%%%%%%%%%%%%%%%%%%%%%%%%%%%%%%%%%%%%%%%%%%%%%%%%%%%%%%%%%%%%%%%%%%%%%%%%%%%%%%%%%%%
\subsection{Linearized equation}
Let us
linearize the system (\re{eq}) at the kink $S(x)$
splitting  the solution as the sum
\be\la{dec}
Y(t)=S+X(t),
\ee
In detail, denote $Y=(\psi,\pi)$ and $X=(\Psi,\Pi)$.
Then (\re{dec}) means that
\be \la{add}
\left\{
\ba{rclrcl}
\psi(x,t)&=&s(x)+\Psi(x,t)\\
\pi(x,t)&=&\Pi(x,t)
\ea
\right.
\ee
Let us substitute (\re{add}) to (\re{eq}), and linearize the equations in $X$.
First,
\be\la{addeq}
\left\{
\ba{rcl}
\dot\Psi(x,t)&=&\Pi(x,t)
\\
\\
\dot\Pi(x,t)
&=&s''(x)+\Psi''(x,t)+ F(s(x)+\Psi(x,t))
\ea
\right.
\ee
%%%%%%%%%%%%%%%%%%%%%%%%%%%%%%%%%%%%%%%%%%%%%%%%%%%%%%%%%%%%%%%%%%%%%%%%%%%%%%%%
Second, 
by (\ref{steq}) we can write the equations (\re{addeq}) as
\be\la{lin}
\dot X(t)=AX(t)+{\cal N}(X(t)),\,\,\,t\in\R
\ee
where  ${\cal N}(X)$ is at least quadratic in $X$.
The linear operator $A$ is
\be\la{AA}
  A=\left(\ba{cc}
     0        &           1 \\

     -H       &           0
  \ea\right)
\ee
where $H$ is the Schr\"odinger operator
\be\la{V}
H=-\ds\frac{d^2}{dx^2}-F'(s)=-\ds\frac{d^2}{dx^2}+m^2+ V(x),
~~~~~~
V(x)=-F'(s(x))-m^2=U''(s(x))-m^2
\ee
Finally,  ${\cal N}(X)$ in (\ref{lin}) is given by
\be\la{N}
{\cal N}(x,X)=\left(\ba{c}
0 \\ N(x,\Psi)
\ea
\right),\quad N(x,\Psi)=F(s(x)+\Psi)-F(s(x))-F'(s(x))\Psi
\ee
%%%%%%%%%%%%%%%%%%%%%%%%%%%%%%%%%%%%%%%%%%%%%%%%%%%%%%%%%%%%%%%%%%%%%%%%%%%%%%%%%
%%%%%%%%%%%%%%%%%%%%%%%%%%%%%%%%%%%%%%%%%%%%%%%%%%%%%%%%%%%%%%%%%%%%%%%%%%%%%%%%%
\subsection{Spectrum of linearized equation}
%%%%%%%%%%%%%%%%%%%%%%%%%%%%%%%%%%%%%%%%%%%%%%%%%%%%%%%%%%%%%%%%%%%%%%%%%%%%%%%%%
%%%%%%%%%%%%%%%%%%%%%%%%%%%%%%%%%%%%%%%%%%%%%%%%%%%%%%%%%%%%%%%%%%%%%%%%%%%%%%%%%
Let us consider the eigenvalue problem  for the operator (\ref{AA}):
$$
A\left(\ba{c} u_{1}\\u_{2}\ea\right)=\left(
\ba{cc}
0& 1 \\
-H & 0\\
\ea
\right)\left(
\ba{c}
u_1 \\ u_{2}
\ea
\right)=\Lam\left(\ba{c}
u_{1} \\ u_{2}
\ea\right)
$$
The first equation implies that $u_{2}=\Lam u_{1}$.
Then  $u_{1}$ satisfies the equation
\be\la{lam}
\Big(H +\Lam^2\Big)u_{1}=0
\ee
Hence, Condition (U2) implies  that the operator $A$ have two purely
imaginary eigenvalues  $\Lam=\pm i\mu$ where $\mu=\sqrt{\lam_1}$.
The corresponding eigenvectors
$$
u=\left(\ba{c} u_1 \\ u_2\ea\right)
=\left(\ba{c} \varphi_{\lam_1} \\ i\mu\varphi_{\lam_1} \ea\right),\quad
\ov u=\left(\ba{c}\varphi_{\lam_1} \\ -i\mu\varphi_{\lam_1}
\ea\right)
$$
where we choose $\varphi_{\lam_1}$ to be real function. This is possible
since $H$ is a differential operator with real coefficients.
The continuous spectrum of the operator $A$ coincides with
${\cal C}:=(-i\infty,-im]\cup[im,i\infty)$. The end points $\Lam=\pm im$
are not eigenvalues nor resonances for the operator $A$ by Condition {\bf U2}.
%%%%%%%%%%%%%%%%%%%%%%%%%%%%%%%%%%%%%%%%%%%%%%%%%%%%%%%%%%%%%%%%%%%%%%%%%%%%%%%%%%%
%%%%%%%%%%%%%%%%%%%%%%%%%%%%%%%%%%%%%%%%%%%%%%%%%%%%%%%%%%%%%%%%%%%%%%%%%%%%%%%%%%%
\subsection{Decay for the linearized dynamics}
%%%%%%%%%%%%%%%%%%%%%%%%%%%%%%%%%%%%%%%%%%%%%%%%%%%%%%%%%%%%%%%%%%%%%%%%%%%%%%%%%%%
%%%%%%%%%%%%%%%%%%%%%%%%%%%%%%%%%%%%%%%%%%%%%%%%%%%%%%%%%%%%%%%%%%%%%%%%%%%%%%%%%%%
Let us consider  the linearized equation
\be\la{lin1}
\dot X(t)=AX(t),\,\,\,t\in\R
\ee
where $A$ is given in (\ref{AA}) with $V$ is defined in (\ref{V}).
Let $\langle\cdot,\cdot\rangle$ be the scalar product in $L^2(\R,\C^2)$.
Denote by $P^d$ the symplectic projector onto the eigenspace $E^d$
generated by $u$ and $\ov u$:
\be\la{PdX}
P^d\,X=\frac{\langle X,ju\rangle}{\langle u,ju\rangle}u
+\frac{\langle X,j\ov u\rangle}{\langle \ov u,j\ov u\rangle}\ov u,
\quad X\in E_\si,\quad\si\in\R
\ee
where
\be\la{j}
 j=\left(\ba{cc} 0 & -1 \\ 1 &  0\ea\right)
\ee
Note, that $P^d X$ is real for real $X$.
Let  $P^c=1-P^d$ be the projector onto the continuous spectrum of operator $A$,
and by $E^c$ the continuous spectral subspace.

Next decay estimates will play the key role in our proofs.
The first estimate follows from  our assumption {\bf U2} by
Theorem 3.9 of  \cite{1dkg} since the condition
of type \cite[(3.1)]{1dkg} holds in our case in the class of the odd functions
(\re{odd}).
%%%%%%%%%%%%%%%%%%%%%%%%%%%%%%%%%%%%%%%%%%%%%%%%%%%%%%%%%%%%%%%%%%%%%%%%%%%%%%%
\begin{pro}\la{1d}
Let the condition {\bf U2} hold, and $\sigma>5/2$.
Then for any $X\in E_\si$ the weighted energy decay holds
\begin{equation}\label{t-as1}
\Vert e^{At} P^c X\Vert_{E_{-\si}}
\le C(1+t)^{-3/2}\Vert X\Vert_{E_{\si}},\; t\in\R
\end{equation}
\end{pro}
%%%%%%%%%%%%%%%%%%%%%%%%%%%%%%%%%%%%%%%%%%%%%%%%%%%%%%%%%%%%%%%%%%%%%%%%%%%%%%%%%
\begin{cor}\la{cr1}
For $\sigma>5/2$ we have for $X\in E_\si\cap W$
\begin{equation}\label{t-as3}
\Vert (e^{At} P^c X)_1\Vert_{L^{\infty}}
\le C(1+t)^{-1/2}(\Vert X\Vert_{W}+\Vert X\Vert_{E_{\si}}),\; t\in\R
\end{equation}
Here $(\cdot)_1$ stands for the first component of the vector function.
\end{cor}
%%%%%%%%%%%%%%%%%%%%%%%%%%%%%%%%%%%%%%%%%%%%%%%%%%%%%%%%%%%%%%%%%%%%%%%%%%%%%%
\begin{proof}
Let us apply the projector $P^c$ to both sides of (\ref{lin1}):
\be\la{lin2}
P^c \dot X=AP^c X=A_0 P^cX-{\cal V}P^c X
\ee
where
\be\la{A0}
A_0=\left(\ba{cc}0& 1 \\\frac{d^2}{dx^2}-m^2 & 0\\\ea\right),\quad
{\cal V}=\left(\ba{cc}0& 0 \\  V & 0\\\ea\right)
\ee
Hence, the Duhamel representation gives,
\be\la{Dug}
  e^{At}P^c X= e^{A_0t}P^c X-
  \int\limits_0^t e^{A_0(t-\tau)}{\cal V}e^{A\tau}P^c Xd\tau, ~~~~t >0.
\ee
Applying estimate (265) from \cite{RS3}, the H\"older inequality
and Proposition \ref{1d} we obtain
$$
\Vert (e^{At} P^c X)_1\Vert_{L^{\infty}}\le C(1+t)^{-1/2}\Vert P^c X\Vert_{W}
+ C\int\limits_0^t (1+t-\tau)^{-1/2}\Vert V(e^{A\tau}P^c X)_1\Vert_{W_0^{1,1}}~d\tau
$$
$$
\le C(1+t)^{-1/2}\Vert X\Vert_{W}
+ C\int\limits_0^t (1+t-\tau)^{-1/2}\Vert e^{A\tau}P^c X\Vert_{E_{-\si}}~d\tau
$$
$$
\le C(1+t)^{-1/2}\Vert X\Vert_{W}
+ C\int\limits_0^t (1+t-\tau)^{-1/2}(1+\tau)^{-3/2}\Vert X\Vert_{E_{\si}}~d\tau
\le C(1+t)^{-1/2}(\Vert X\Vert_{W}+\Vert X\Vert_{E_{\si}})
$$
\end{proof}
%%%%%%%%%%%%%%%%%%%%%%%%%%%%%%%%%%%%%%%%%%%%%%%%%%%%%%%%%%%%%%%%%%%%%%%%%%%%%
\bp\la{Amu}
For $\si>5/2$ the bounds hold
\be\la{t-as2}
\Vert e^{At}(A\mp 2i\mu-0)^{-1} P^c X\Vert_{E_{-\si}}
\le C(1+t)^{-3/2}\Vert X\Vert_{E_{\si}},\; t>0
\ee
\ep
%%%%%%%%%%%%%%%%%%%%%%%%%%%%%%%%%%%%%%%%%%%%%%%%%%%%%%%%%%%%%%%%%%%%%%%%%%%%%%
We will prove the proposition in Appendix B.
%%%%%%%%%%%%%%%%%%%%%%%%%%%%%%%%%%%%%%%%%%%%%%%%%%%%%%%%%%%%%%%%%%%%%%%%%%%%%%%%%%%%%%
%%%%%%%%%%%%%%%%%%%%%%%%%%%%%%%%%%%%%%%%%%%%%%%%%%%%%%%%%%%%%%%%%%%%%%%%%%%%%%%%%%%%

\section{Decomposition of the dynamics}
\label{modsec}
%%%%%%%%%%%%%%%%%%%%%%%%%%%%%%%%%%%%%%%%%%%%%%%%%%%%%%%%%%%%%%%%%%%%%%%%%%%%%%%%%%%%%%
We decompose the  solution to (\ref{Eq}) as $Y(t)=S+X(t)$, where
$X(t)=w(t)+f(t)$ with
$w(t)=z(t)u+\ov z(t)\ov u\in E^d$ and $f(t)\in E^c$.

\begin{lemma}\la{mod}
Let $Y(t)=S+w(t)+f(t)$ be a solution to the Cauchy problem (\re{Eq}).
Then the functions $z(t)$  and $f(t)$ satisfy the equations
\be\la{zmod}
(\dot z-i\mu z)\langle u,ju\rangle=\langle{\cal N},ju\rangle
\ee
\be\la{fmod}
\dot f=Af+P^c{\cal N}
\ee
with ${\cal N}$ defined in (\ref{N}).
\end{lemma}
\begin{proof}
Applying the projector $P^d$ to the equation (\ref{lin}), we obtain
\be\la{z}
\dot zu+\dot{\ov z}\,{\ov u}=Aw+P^d{\cal N}.
\ee
Using $\langle\ov u,ju\rangle=0$ and $Aw=i\mu(zu-\ov z\,\ov u)$,
 we get equation (\ref{zmod}), after taking the scalar product
of equation (\ref{z}) with $ju$ since $(P^d)^*j=jP^d$.
Applying $P^c$ to (\re{lin}), we obtain (\ref{fmod}) since $P^c$ commutes with $A$.
\end{proof}
%%%%%%%%%%%%%%%%%%%%%%%%%%%%%%%%%%%%%%%%%%%%%%%%%%%%%%%%%%%%%%%%%%%%%%%%%%%%%%%%%%%%
\br\la{order}
In the remaining part of the paper we shall prove the following asymptotics
\be\la{exas}
  \Vert f(t)\Vert_{E_{-\si}}\sim t^{-1}, \quad
  z(t)\sim t^{-1/2},\quad \Vert f_1(t)\Vert_{L^{\infty}}\sim t^{-1/2}, \quad t\to\infty
\ee
To justify these asymptotics, we will single out leading terms
in right hand side of the equations (\ref{zmod})-(\ref{fmod}).
Namely, we shall expand  the expressions  for
$\dot z$ up to terms of the order ${\cal O}(t^{-3/2})$,
and for $\dot f$ up to  ${\cal O}(t^{-1})$ keeping in mind the asymptotics
(\re{exas}). This choice allow us to obtain the uniform bounds using
the method of majorants.
\er

Now let us expand $N(x,\Psi)$ from (\ref{N}) in the Taylor series
\be\la{Njj}
N(x,\Psi)=N_2(x,\Psi)+...+N_{12}(x,\Psi)+N_R(x,\Psi)
\ee
where
\be\la{cNj}
N_j(x,\Psi)=\fr{F^{(j)}(s(x))}{j!} \Psi^j,\quad j=2,...,12
\ee
and $N_R$ is the remainder.
Condition {\bf U1} implies that 
$F(\psi)=
-m^2(\psi\mp a)+
\cO(|\psi\mp a|^{2k-2})$
as $\psi\to\pm a$ where $k=7$ by our assumption.
Hence, the functions
$N_j(x,\Psi)$ with $j\le 12$ decrease exponentially  as $|x|\to\infty$
by (\ref{s-decay}), while for the remainder $N_R$ we have
\be\la{NRR}
|N_R|={\cal R}(|\Psi|)|\Psi|^{13}
={\cal R}(|z|+\Vert f_1\Vert_{L^{\infty}})|\Psi|^{13}
\ee
where ${\cal R}(A)$ is a general notation for a positive function
which remains bounded as $A$  is sufficiently small.

Let us define ${\cal N}_2[X_1,X_2]=(0,N_2[\Psi_1,\Psi_2])$ and
${\cal N}_3[X_1,X_2,X_3]=(0,N_3[\Psi_1,\Psi_2,\Psi_3])$
as the symmetric bilinear and trilinear forms with
\be\la{NN}
N_2[\Psi_1,\Psi_2]=\fr{F''(s)}{2} \Psi_1\Psi_2,\quad
N_3[\Psi_1,\Psi_2,\Psi_3]=\fr{F'''(s)}{6}\Psi_1\Psi_2\Psi_3
\ee
%%%%%%%%%%%%%%%%%%%%%%%%%%%%%%%%%%%%%%%%%%%%%%%%%%%%%%%%%%%%%%%%%%%%%%%%%%%%%%%%

%%%%%%%%%%%%%%%%%%%%%%%%%%%%%%%%%%%%%%%%%%%%%%%%%%%%%%%%%%%%%%%%%%%%%%%%%%%%%%%%%%%%
%\section{Leading terms in dynamical equations}
%\setcounter{equation}{0}
%\label{leadterms}
%%%%%%%%%%%%%%%%%%%%%%%%%%%%%%%%%%%%%%%%%%%%%%%%%%%%%%%%%%%%%%%%%%%%%%%%%%%%%%%%%%%%%
%%%%%%%%%%%%%%%%%%%%%%%%%%%%%%%%%%%%%%%%%%%%%%%%%%%%%%%%%%%%%%%%%%%%%%%%%%%%%%%%%%
%%%%%%%%%%%%%%%%%%%%%%%%%%%%%%%%%%%%%%%%%%%%%%%%%%%%%%%%%%%%%%%%%%%%%%%%%%%%%%%%%%%%
\subsection{Leading term in $\dot z$}
\label{lt-3}
%%%%%%%%%%%%%%%%%%%%%%%%%%%%%%%%%%%%%%%%%%%%%%%%%%%%%%%%%%%%%%%%%%%%%%%%%%%%%%%%%%%%%
Let us  rewrite (\ref{zmod}) in the form:
\be\la{z1}
  \dot z-i\mu z=\frac{\langle {\cal N},ju\rangle}{\langle u,ju\rangle}
  =\frac{\langle {\cal N}_2[w,w]+2{\cal N}_2[w,f]+
   {\cal N}_3[w,w,w],ju\rangle}{\langle u,ju\rangle}+Z_R
\ee
where
\be\la{ZR}
  |Z_R|={\cal R}(|z|+\Vert f_1\Vert_{L^{\infty}})
  (|z|^2+\Vert f\Vert_{E_{-\si}})^2
\ee
Note that
\be\la{N2uu}
{\cal N}_2[w,w]=(z^2+2z\ov z+\ov z^2){\cal N}_2[u,u],\quad
{\cal N}_3[w,w,w]=(z^3+3z^2\ov z+3z\ov z^2+\ov z^3){\cal N}_3[u,u,u]\quad
\ee
Hence,  (\ref{z1}) reads
\be\la{z2}
  \dot z=i\mu z+Z_{2}(z^2+2z\ov z+\ov z^2)
  +Z_3(z^3+3z^2\ov z+3z\ov z^2+\ov z^3)
  +(z+\ov z)\langle f,jZ'_{1}\rangle+Z_R
  \ee
where, by (\ref{NN}),
\be\la{Zij}
  Z_{2}=\frac{\langle {\cal N}_2[u,u],ju\rangle}{\langle u,ju\rangle},\quad
  Z_{3}=\frac{\langle {\cal N}_3[u,u,u],ju\rangle}{\langle u,ju\rangle},\quad
  Z'_{1}=2\frac{{\cal N}_2[u,u]}{\langle u,ju\rangle}
\ee

%%%%%%%%%%%%%%%%%%%%%%%%%%%%%%%%%%%%%%%%%%%%%%%%%%%%%%%%%%%%%%%%%%%%%%%%%%%%%%%%%%%%
%%%%%%%%%%%%%%%%%%%%%%%%%%%%%%%%%%%%%%%%%%%%%%%%%%%%%%%%%%%%%%%%%%%%%%%%%%%%%%%%%%%%
\subsection{Leading term in $\dot f$}
\label{lt-4}
%%%%%%%%%%%%%%%%%%%%%%%%%%%%%%%%%%%%%%%%%%%%%%%%%%%%%%%%%%%%%%%%%%%%%%%%%%%%%%%%%%%%%
%%%%%%%%%%%%%%%%%%%%%%%%%%%%%%%%%%%%%%%%%%%%%%%%%%%%%%%%%%%%%%%%%%%%%%%%%%%%%%%%%%%%
We now turn to equation (\ref{fmod}) which we rewrite in the form
\be\la{f}
  \dot f=Af+P^c{\cal N}=Af+P^c{\cal N}_2[w,w]+F_R
\ee
The remainder $F_R=F_R(x,t)$ reads
\be\la{FR}
  F_R=P^c({\cal N}(X)-{\cal N}_2[w,w])=
  (1-P^d)({\cal N}(X)-{\cal N}_2[w,w])=F_{I}+F_{II}+F_{III}
\ee
Here
\be\la{FR1}
F_I=-P^d({\cal N}(X)-{\cal N}_2[w,w]),
\quad F_{II}={\cal N}(X)-{\cal N}_2[w,w]-{\cal N}_R,
\quad F_{III}={\cal N}_R
\ee
where ${\cal N}_R =(0,N_R)$ with $N_R$ defined in (\ref{Njj}).
For $F_I$ and $F_{II}$ the following bound holds
\be\la{FI}
\Vert F_{I}+F_{II}\Vert_{E_{\si}\cap W}={\cal R}(|z|+\Vert f_1\Vert_{L^{\infty}})
(|z|^3+|z|\Vert f\Vert_{E_{-\si}}+\Vert f_1\Vert_{L^{\infty}}\Vert f\Vert_{E_{-\si}})
\ee
Indeed, $F_I$ admits the estimate by (\ref{PdX}) since the function $u(x)$
decays exponentially. Further,
$$
(F_{II})_1=0,~~(F_{II})_2=2N_2(w_1,f_1)+N_2(f_1,f_1)+N_3(\Psi)+...+N_{12}(\Psi)
$$
and each summand contains an exponentially decreasing factor by
{\bf U1}, (\ref{s-decay}) and (\ref{cNj}).
Similarly,
\be\la{F2-est}
\Vert P^c{\cal N}_2[w,w]\Vert_{E_{\si}\cap W}\le C|z|^2
\ee
It remains to estimate the term  $F_{III}$.
%%%%%%%%%%%%%%%%%%%%%%%%%%%%%%%%%%%%%%%%%%%%%%%%%%%%%%%%%%%%%%%%%%%%%%%%%%%%%%%%%
\begin{lemma}\la{FIII-est}
For  $0<\nu<1/2$, the term  $F_{III}={\cal N}_R=(0,N_R)$ admits the  estimate
\be\la{FIII}
\Vert F_{III}\Vert_{E_{5/2+\nu}}={\cal R}(|z|+\Vert f_1\Vert_{L^{\infty}})
(1+t)^{4+\nu}(|z|^{12}+\Vert f_1\Vert_{L^{\infty}}^{12})
\ee
\end{lemma}
\begin{proof}
Estimate (\ref{FIII}) means that
\be\la{jv2}
\Vert N_R\Vert_{L^2_{5/2+\nu}}={\cal R}(|z|+\Vert f_1\Vert_{L^{\infty}})
(1+t)^{4+\nu}(|z|^{12}+\Vert f_1\Vert_{L^{\infty}}^{12})
\ee
By (\ref{NRR}), we have
$$
\Vert N_R\Vert_{L^2_{5/2+\nu}}={\cal R}(|z|+\Vert f_1\Vert_{L^{\infty}})
(|z|^{12}+\Vert f_1\Vert_{L^{\infty}}^{12})\Vert \Psi\Vert_{L^2_{5/2+\nu}}
$$
We will prove  in Appendix A that
\be\la{jv1}
\Vert\Psi(t)\Vert_{L^2_{5/2+\nu}}\le C(d_0)(1+t)^{4+\nu}
\ee
Then (\ref{jv2}) follows.
\end{proof}
%%%%%%%%%%%%%%%%%%%%%%%%%%%%%%%%%%%%%%%%%%%%%%%%%%%%%%%%%%%%%%%%%%%%%%%%%%%%%
\begin{lemma}\la{FIIIW-est}
The bound holds
\be\la{FIIIW}
\Vert F_{III}\Vert_W={\cal R}(|z|+\Vert f_1\Vert_{L^{\infty}})
(|z|^{10}+\Vert f_1\Vert_{L^{\infty}}^{10})
\ee
\end{lemma}
%%%%%%%%%%%%%%%%%%%%%%%%%%%%%%%%%%%%%%%%%%%%%%%%%%%%%%%%%%%%%%%%%%%%%%%%%%%%%%%%
\begin{proof}
{\it Step i)} By the Cauchy formula,
$$
\ti N_{R}(x,t)=N_{12}(x,t)+N_{R}(x,t)=\frac{\Psi^{12}(x,t)}{(12)!}
\int\limits_0^1(1-\rho)^{11}F^{(12)}(s+\rho\Psi(x,t))d\rho
$$
Therefore,
$$
\Vert \ti N_{R}\Vert_{L^1}={\cal R}(|z|+\Vert f_1\Vert_{L^{\infty}})\int|\Psi|^{12}dx
={\cal R}(|z|+\Vert f_1\Vert_{L^{\infty}})(|z|+\Vert f_1\Vert_{L^{\infty}})^{10}
\Vert\Psi\Vert^2_{L^2}
$$
$$
={\cal R}(|z|+\Vert f_1\Vert_{L^{\infty}})(|z|^{11}+\Vert f_1\Vert_{L^{\infty}}^{10})
$$
since $\Vert\Psi(t)\Vert_{L^2}\le C(d_0)$ by the results of \cite{HPW}.\\
{\it Step ii)}
Further,
$$
\ti N'_{R}=\frac{\Psi^{12}}{(12)!}\int\limits_0^1(1\!-\!\rho)^{11}(s'+\rho\Psi')
F^{(13)}(s+\rho\Psi)d\rho
+\frac{\Psi^{11}\Psi'}{(11)!}\!\int\limits_0^1(1\!-\!\rho)^{11}F^{(12)}
(s+\rho\Psi)d\rho
$$
Therefore,
$$
\Vert \ti N'_{R}\Vert_{L^1}={\cal R}(|z|+\Vert f_1\Vert_{L^{\infty}})
(|z|+\Vert f_1\Vert_{L^{\infty}})^{10}\int|\Psi(x)\Psi'(x)|dx
\le{\cal R}(|z|+\Vert f_1\Vert_{L^{\infty}})
(|z|^{10}+\Vert f_1\Vert_{L^{\infty}}^{10})
$$
since
$\ds\int|\Psi(x)\Psi'(x)|dx\le\Vert\Psi\Vert_{L^2}\Vert\Psi'\Vert_{L^2}\le C(d_0)$.
Finally, let us note that
$$
\Vert N_{12}\Vert_{W_0^{1,1}}\le{\cal R}(|z|+\Vert f_1\Vert_{L^{\infty}})
(|z|^{10}+\Vert f_1\Vert_{L^{\infty}}^{10})
$$
\end{proof}

%%%%%%%%%%%%%%%%%%%%%%%%%%%%%%%%%%%%%%%%%%%%%%%%%%%%%%%%%%%%%%%%%%%%%%%%%%%%%%%%%
%%%%%%%%%%%%%%%%%%%%%%%%%%%%%%%%%%%%%%%%%%%%%%%%%%%%%%%%%%%%%%%%%%%%%%%%%%%%%%%%%%%%
\section{Poincare normal forms}
 \label{eqns-trans}
%%%%%%%%%%%%%%%%%%%%%%%%%%%%%%%%%%%%%%%%%%%%%%%%%%%%%%%%%%%%%%%%%%%%%%%%%%%%%%%%%%%%%
Our goal is to transform the evolution equations for $z$ and $f$
to a ``normal form'' removing the ``nonresonant terms''.
%%%%%%%%%%%%%%%%%%%%%%%%%%%%%%%%%%%%%%%%%%%%%%%%%%%%%%%%%%%%%%%%%%%%%%%%%%%%%%%%%%%%
\subsection{Normal form for $\dot f$}
\label{f-trans}
%%%%%%%%%%%%%%%%%%%%%%%%%%%%%%%%%%%%%%%%%%%%%%%%%%%%%%%%%%%%%%%%%%%%%%%%%%%%%%%%%%%%%

We rewrite (\ref{f}) in a more detailed form as
\be\la{f1}
  \dot f=Af+(z^2+2z\overline z+\overline z^2)F_{2}+F_R,
  \quad  F_{2}=P^c{\cal N}_2[u,u].
\ee
We want to extract from $f$ the term of order $z^2\sim t^{-1}$
(see Remark \ref{order}). For this purpose we expand $f$ as
\be\la{h-dec}
  f=h+k+g,
\ee
where
\be\la{k}
  g(t)=-e^{At}k(0),\quad
  k=a_{20}z^2+2a_{11}z\overline z+a_{02}\overline z^2,
\ee
with some $a_{ji}\equiv a_{ij}(x)$ satisfying $a_{ij}(x)=\overline a_{ij}(x)$.
Note that $h(0)=f(0)$.
%%%%%%%%%%%%%%%%%%%%%%%%%%%%%%%%%%%%%%%%%%%%%%%%%%%%%%%%%%%%%%%%%%%%%%%%%%%%
\begin{lemma}\label{h-trans}
  There  exist the coefficients $a_{ij}\in H^s_{-\si}$ with any $s>0$
  such that the equation for $h_1$ has the form
  \be\la{h5}
    \dot h=Ah+H_R
  \ee
  where
  $$
    H_R=F_R+H_I,\quad {with}\quad H_I=\sum a_{ij}(x)
    {\cal R}(|z|+\Vert f\Vert_{E_{-\si}})|z|(|z|+\Vert f\Vert_{E_{-\si}})^2
   $$
\end{lemma}
%%%%%%%%%%%%%%%%%%%%%%%%%%%%%%%%%%%%%%%%%%%%%%%%%%%%%%%%%%%%%%%%%%%%%%%%%%%%%%
\begin{proof}
  Substituting (\ref{k}) into (\ref{f1}), we get
  \beqn\nonumber
    \dot h&=&\dot f-(2a_{20}z+ 2a_{11}\overline z)\dot z
    -(2a_{11}z+ 2a_{02}\overline z)\dot {\overline z}-\dot g\\
    \nonumber
    &=& Af+(z^2+2z\overline z+\overline z^2)F_2+F_R\\
    \nonumber
    &-&(2a_{20}z+ 2a_{11}\overline z)(i\mu z
    +{\cal R}(|z|+\Vert f\Vert_{E_{-\si}})
    (|z|+\Vert f\Vert_{E_{-\si}})^2)\\
    \nonumber
    &-&(2a_{11}z+ 2a_{02}\overline z)(-i\mu\overline z
    +{\cal R}(|z|+\Vert f\Vert_{E_{-\si}})
    (|z|+\Vert f\Vert_{E_{-\si}})^2)-Ag
  \eeqn
On the other hand, (\ref{h5}) means that
$$
    \dot h=A(f-a_{20}z^2-2a_{11}z\overline z-a_{02}\overline z^2-g)+H_R
$$
%%%%%%%%%%%%%%%%%%%%%%%%%%%%%%%%%%%%%%%%%%%%%%%%%%%%%%%%%%%%%%%%%%%%%%%%%%%%%%%%%%%%%%%%%%%%
 Equating the coefficients of the quadratic powers of $z$, we get
   \beqn\nonumber
     F_{2}-2i\mu a_{20}&=&-Aa_{20}\\
     \nonumber
     F_{2}&=&-A a_{11}\\
   \nonumber
     F_{2}+2i\mu a_{02}&=&-Aa_{02}
   \eeqn
   and
   $$
    H_R=F_R+\sum a_{ij}{\cal R}(|z|+\Vert f\Vert_{E_{-\si}})
    |z|(|z|+\Vert f\Vert_{E_{-\si}})^2
   $$
  Notice that $F_{2}\in E^c$ is smooth, exponentially decreasing function.
  Hence, there exists a solution  $a_{11}$ in the form
  \be\la{a11}
    a_{11}=-A^{-1}F_{2}
  \ee
  where $A^{-1}$ stands for regular part of the resolvent $R(\lam)$
  at $\lam=0$ since the singular part of $R(\lam)F_{2}$ vanishes for $F_{2}\in E^c$.
  The function $a_{11}$ is exponentially decreasing at infinity.

  For $a_{20}$ and $a_{02}$ we
choose  the following inverse operators:
  \be\la{a20}
     a_{20}=-(A-2i\mu-0)^{-1}F_{2},\quad
     a_{02}={\overline a}_{20}=-(A+2i\mu-0)^{-1}F_{2}
  \ee
  This choice  is motivated by Lemma \ref{Amu}.
  \\
  The remainder $H_I$ can be written as
  \be\la{Am}
     H_I=\sum\limits_m(A-2i\mu m-0)^{-1}C_m,\quad m\in\{-1,0,1\}
  \ee
  with   $C_m\in E^c$, satisfying the estimate
  \be\la{Cm-est1}
    \Vert C_{m}\Vert_{E_{\si}}={\cal R}(|z|+\Vert f\Vert_{E_{-\si}})
    |z|(|z|+\Vert f\Vert_{E_{-\si}})^2
  \ee
\end{proof}
%%%%%%%%%%%%%%%%%%%%%%%%%%%%%%%%%%%%%%%%%%%%%%%%%%%%%%%%%%%%%%%%%%%%%%%%%%%%%%%%%%%%
\subsection{Normal form for $\dot z$}
\label{z-transform}
%%%%%%%%%%%%%%%%%%%%%%%%%%%%%%%%%%%%%%%%%%%%%%%%%%%%%%%%%%%%%%%%%%%%%%%%%%%%%%%%%%%%%
Let us consider the equation (\ref{z2}) for $z$. Substituting
 (\ref{h-dec}) into (\ref{z2}) and putting the contribution of
$f=h+k+g$ into the remainder $Z_R$, we obtain
\be\la{z3}
  \dot z=i\mu z+Z_{2}(z^2+2z\overline z+\overline z^2)
  +Z_{3}(z^3+3z^2\overline z+3z\overline z^2+\overline z^3)
  +Z_{30}'z^3+Z_{21}'z^2\overline z+Z_{12}'z{\overline z}^2
  +Z_{03}'{\overline z}^3+\tilde Z_R
\ee
We have by (\ref{h-dec})-(\ref{k})
\be\la{Z3}
  Z_{30}'=\langle a_{20},jZ'_{1}\rangle,\quad
  Z_{21}'=\langle a_{11}+a_{20},jZ'_{1}\rangle,\quad
  Z_{03}'=\langle a_{02},jZ'_{1}\rangle,\quad
  Z_{12}'=\langle a_{02}+a_{11},jZ'_{1}\rangle
\ee
We are specially  interested in the resonance term
$Z'_{21}z^2\ov z=Z'_{21}|z|^2z$.
Formulas (\ref{Zij}), (\ref{a11}), (\ref{a20}) imply
\be\la{Z21}
  Z'_{21}=-\langle A^{-1}P^c{\cal N}_2[u,u],
  2j\frac{{\cal N}_2[u,u]}{\langle u,ju\rangle}\rangle
  -\langle (A-2i\mu-0)^{-1}P^c{\cal N}_2[u,u],
 2j\frac{{\cal N}_2[u,u]}{\langle u,ju\rangle}\rangle
\ee
%%%%%%%%%%%%%%%%%%%%%%%%%%%%%%%%%%%%%%%%%%%%%%%%%%%%%%%%%%%%%%%%%%%%%%%%%%%%%%%
For the $\langle u,ju\rangle$ we get
\be\la{uu}
\langle u,ju\rangle=i\de,\quad {with}\quad\de>0
\ee
%%%%%%%%%%%%%%%%%%%%%%%%%%%%%%%%%%%%%%%%%%%%%%%%%%%%%%%%%%%%%%%%%%%%%%%%%%%%%%%
Now we can prove
\begin{lemma}\label{pv}
  Let the non-degeneracy condition {\bf U3} hold. Then
  \be\la{ReZ}
    \Re Z'_{21}<0
  \ee
\end{lemma}
%%%%%%%%%%%%%%%%%%%%%%%%%%%%%%%%%%%%%%%%%%%%%%%%%%%%%%%%%%%%%%%%%%%%%%%%%%%%%%%%%%
\begin{proof}
  We first notice that the coefficient
  $\langle A^{-1}P^cj{\cal N}_2[u,u],2{\cal N}_2[u,u]\rangle$
  that appears in the expression (\ref{Z21}) for $Z'_{21}$ is real
  since operator $A^{-1}2P^cj$ is selfadjoint.
  Hence  (\ref{uu}) implies that $\Re Z'_{21}$ reduces to
$$
  \Re Z'_{21}=\Re 2\frac{\langle (A-2i\mu-0)^{-1}P^c
  {\cal N}_2[u,u], j{\cal N}_2[u,u]\rangle}{i\de}
  =\fr 2{\de}\Im\langle R(2i\mu+0)P^c{\cal N}_2[u,u],j{\cal N}_2[u,u]\rangle
$$
  where we denote
$$
   R(\lam)=(A-\lam)^{-1},\; \Re\lam>0
$$
  Using that $P^c$ commutes with $ R(2i\mu+0)$, we have
  $ R(2i\mu+0)P^c=P^c R(2i\mu+0)P^c$.
  We have also that $(P^c)^*j=jP^c$, hence
  $$
    \Re Z'_{21}=\frac 2{\de}\Im\langle  R(2i\mu+0)\al, j\al\rangle
  $$
  with $\al=P^c{\cal N}_2[u,u]$.
  Now we use the representation (cf.\cite{BS}, formula(2.1.9))
  \beqn\la{R-rep}
  \langle  R(2i\mu+0)\al, j\al\rangle \!\!\!&=&\!\!\!\fr 1i\int\limits_{b}^{\infty}
  \theta(\lam)d\lam\Big(\fr{\langle\al,ju(i\lam)\rangle\langle u(i\lam),j\al\rangle}
  {i\lam-2i\mu-0}+\fr{\langle\al,j{\ov u}(i\lam)\rangle\langle{\ov u}(i\lam),j\al\rangle}
  {-i\lam-2i\mu-0}\Big)\\
  \nonumber
  \!\!\!&=&\!\!\!\int\limits_{b}^{\infty}
  \theta(\lam)d\lam\Big(\fr{\langle u(i\lam),j\al\rangle\ov{\langle u(i\lam),j\al\rangle}}
  {\lam-2\mu+i0}+\fr{\langle{\ov u}(i\lam),j\al\rangle\ov{\langle{\ov u}(i\lam),j\al\rangle}}
  {\lam+2\mu-i0}\Big)\\
  \eeqn
  Using that $\ds\fr 1{\nu+i0}={\rm p.v.}\ds\fr 1{\nu}-i\pi\delta(\nu)$,
  where p.v. is the Cauchy principal value, we obtain
  \be\la{Cpv}
   \langle  R(2i\mu+0)\al, j\al\rangle=
   \int\limits_{\sqrt{2}}^{\infty}
  \theta(\lam)d\lam\Big(\fr{|\langle u(i\lam),j\al\rangle|^2}{\lam-2\mu}
  +\fr{|\langle{\ov u}(i\lam),j\al\rangle|^2}{\lam+2\mu}\Big)
  -i\pi\theta(2\mu)|\langle u(2i\mu),j\al\rangle|^2
  \ee
  The integral terms in (\ref{Cpv}) is real. Thus,
  $$
    \Im\langle  R_T(2i\mu_T+0)\al, j\al\rangle=
     -\pi\theta(2\mu)|\langle u(2i\mu),j\al\rangle|^2<0
  $$
  since $\theta(2\mu)>0$, and
  $$
  \langle u(2i\mu),j\al\rangle= \langle u(2i\mu),jP^c{\cal N}_2[u,u]\rangle
  =\langle u(2i\mu),j{\cal N}_2[u,u]\rangle=-\int u_1(2i\mu)(x)N_2[u,u](x)dx
  $$
  $$
  =-\fr 12\int \varphi_{4\lam_1}(x)F''(s(x))\varphi_{\lam_1}^2(x)dx
  \not =0
  $$
  by {\bf U3}.
\end{proof}
%%%%%%%%%%%%%%%%%%%%%%%%%%%%%%%%%%%%%%%%%%%%%%%%%%%%%%%%%%%%%%%%%%%%%%%%%%%%%%%%%%%

Further we need an estimate for the remainder $\tilde Z_R$.
\begin{lemma}\label{tZR}
  The remainder $\tilde Z_R$ has the form
  \be\la{tZRest}
    |\tilde Z_R|={\cal R}_1(|z|+\Vert f\Vert_{L^{\infty}})
    \Bigl[(|z|^2+\Vert f\Vert_{E_{-\si}})^2
   +|z|\Vert g\Vert_{E_{-\si}}+|z|\Vert h\Vert_{E_{-\si}}\Bigr]
  \ee
\end{lemma}
{\bf Proof}
  The remainder $\tilde Z_R$ is given by
  $$
    \tilde Z_R=Z_R+(z+\overline z)\langle f-k,jZ'_{1}\rangle
  $$
   where $Z_R$ satisfies estimate (\ref{ZR}).
  Since $f-k=g+h$, we have
 $$
    |\langle f-k,Z'_{1}\rangle|\le
    C(\Vert g\Vert_{E_{-\si}}+\Vert h\Vert_{E_{-\si}})
 $$
Hence, (\ref{tZRest}) follows.
Now we can apply the Poincar\'e method of normal coordinates to equation (\ref{z3}).
%%%%%%%%%%%%%%%%%%%%%%%%%%%%%%%%%%%%%%%%%%%%%%%%%%%%%%%%%%%%%%%%%%%%%%%%%%%%%%%%%%%%
\begin{lemma}\label{nc}  (cf. \cite[Proposition 4.9]{BS})
  There exist coefficients $c_{ij}$ such that the new function $z_1(t)$ defined by
  $$
    z_1=z+c_{20}z^2+c_{11}z\ov z+c_{02}\ov z^2+c_{30}z^3
    +c_{12}z\ov z^2+c_{03}\ov z^3
  $$
  satisfies an equation of the form
  \be\la{z4}
    \dot z_1=i\mu z_1+iK|z_1|^2z_1+\hat Z_R
  \ee
  where $\hat Z_R$ satisfies estimates of the same type as $\tilde Z_R$, and
  \be\la{Re K}
    \Re~ iK=\Re Z'_{21}<0
  \ee
\end{lemma}
\begin{proof}
  Substituting $z_1$ in equation (\ref{z3}) for $z$
  and equating the coefficients, we get, in particular,
  \be\la{cij}
     c_{20}=\frac{i}{\mu}Z_{2},\quad c_{11}=-\frac{2i}{\mu}Z_{2},
     \quad c_{02}=-\frac{i}{3\mu}Z_{2}
   \ee
  and
  \be\la{K}
     iK=3Z_{3}+ Z_{21}'+(4c_{20}-c_{11}-2c_{20})Z_{2}
  \ee
  Since the coefficients $Z_{2}$ and $Z_{3}$ defined in (\ref{Zij}) are
  purely imaginary then (\ref{Re K}) is follow.
\end{proof}
%%%%%%%%%%%%%%%%%%%%%%%%%%%%%%%%%%%%%%%%%%%%%%%%%%%%%%%%%%%%%%%%%%%%%%%%%%%%%%%%%%%%%%%%
It is easier to deal with $y=|z_1|^2$ rather that $z_1$ because $y$ decreases
at infinity while $z_1$ is oscillating.
The equation satisfied by $y$ is simply obtained by multiplying (\ref{z4}) by $\ov z_1$
and taking the real part:
\be\la{y}
  \dot y=2\Re(iK)y^2 +Y_R,
\ee
where
\be\la{YR}
  |Y_R|={\cal R}_1(|z|+\Vert f\Vert_{L^{\infty}})|z|
  \Bigl[(|z|^2+\Vert f\Vert_{E_{-\si}})^2+ |z|\Vert g\Vert_{E_{-\si}}
  +|z|\Vert h\Vert_{E_{-\si}}\Bigr].
\ee
%%%%%%%%%%%%%%%%%%%%%%%%%%%%%%%%%%%%%%%%%%%%%%%%%%%%%%%%%%%%%%%%%%%%%%%%%%%%%%%%%
%%%%%%%%%%%%%%%%%%%%%%%%%%%%%%%%%%%%%%%%%%%%%%%%%%%%%%%%%%%%%%%%%%%%%%%%%%%%%%%%%%%%
\subsection{Summary of normal forms}
\label{sum-eq}
%%%%%%%%%%%%%%%%%%%%%%%%%%%%%%%%%%%%%%%%%%%%%%%%%%%%%%%%%%%%%%%%%%%%%%%%%%%%%%%%%%%%%
%%%%%%%%%%%%%%%%%%%%%%%%%%%%%%%%%%%%%%%%%%%%%%%%%%%%%%%%%%%%%%%%%%%%%%%%%%%%%%%%%%%
We summarize the main formulas of Sections \ref{f-trans}-\ref{z-transform}.
First we recall that
$$
  f=k+g+h
$$
where $k$ and $g$ are defined in (\ref{k}).
The equation satisfied by $f$ and $h$ are respectively (see (\ref{f}) and (\ref{h5}))
\be\la{f-dot}
\dot f=Af+\tilde F_R,
\ee
\be\la{h-dot}
  \dot h=A h+H_R
\ee
Here $\tilde F_R=P^c{\cal N}_2[w,w]+F_R$, $F_R=F_{I}+F_{II}+F_{III}$, $H_R=F_R+H_I$.
The remainders $F_{I}$, $F_{II}$, $P^c{\cal N}_2[w,w]$ and $F_{III}$
are estimated in (\ref{FI})-(\ref{FIII}), (\ref{FIIIW}).
The remainder $H_I$ is estimated in (\ref{Am}) and (\ref{Cm-est1}).
Note, that
\be\la{fh}
\Vert f\Vert_{E_{-\si}}
\le C(\Vert g\Vert_{E_{-\si}}+|z|^2+\Vert h\Vert_{E_{-\si}})
\ee
\\
%%%%%%%%%%%%%%%%%%%%%%%%%%%%%%%%%%%%%%%%%%%%%%%%%%%%%%%%%%%%%%%%%%%%%%%%%%%
The second equation describes the evolution of $z_1$ from (\ref{z4}):
\be\la{z-s}
  \dot z_1=i\mu z_1+iK|z_1|^2z_1+\hat Z_R
\ee
where the remainder $\hat Z_R$ admits the
estimate (\ref{tZRest}).
From (\ref{z3}), $z$ and $z_1$ are related by
\be\la{z-rel}
  z_1-z={\cal O}|z|^2.
\ee
The fourth equation is the dynamics for $y=|z_1|^2$
\be\la{y-s}
  \dot y=2\Re(iK)y^2 +Y_R,
\ee
where the remainder $Y_R$ admits the
estimate (\ref{YR}).
Here $\Re iK<0$ by Lemma \ref{pv}
that is the key point.

%%%%%%%%%%%%%%%%%%%%%%%%%%%%%%%%%%%%%%%%%%%%%%%%%%%%%%%%%%%%%%%%%%%%%%%%%%%%%%%%%
%%%%%%%%%%%%%%%%%%%%%%%%%%%%%%%%%%%%%%%%%%%%%%%%%%%%%%%%%%%%%%%%%%%%%%%%%%%%%%%%%%%%
\section{Majorants}
 \label{maj}
%%%%%%%%%%%%%%%%%%%%%%%%%%%%%%%%%%%%%%%%%%%%%%%%%%%%%%%%%%%%%%%%%%%%%%%%%%%%%%%%%%%%%
%%%%%%%%%%%%%%%%%%%%%%%%%%%%%%%%%%%%%%%%%%%%%%%%%%%%%%%%%%%%%%%%%%%%%%%%%%%%%%%%%%%%
\subsection{Notations}
\label{maj-def}
%%%%%%%%%%%%%%%%%%%%%%%%%%%%%%%%%%%%%%%%%%%%%%%%%%%%%%%%%%%%%%%%%%%%%%%%%%%%%%%%%%%%%
We define the 'majorants'
\be\la{M1}
   \M_1(T)=\max\limits_{0\le t\le T}|z(t)|\Biggl(\frac{\ve}{1+\ve t}\Biggr)^{-1/2}
\ee
\be\la{M2}
   \M_2(T)=\max\limits_{0\le t\le T}\Vert f_1(t)\Vert_{L^{\infty}}
   \Biggl(\frac{\ve}{1+\ve t}\Biggr)^{-1/2}\log^{-1}(2+\ve t)
\ee
\be\la{M3}
   \M_3(T)=\max\limits_{0\le t\le T}\Vert h(t)\Vert_{E_{-5/2-\nu}}
   \Biggl(\frac{\ve}{1+\ve t}\Biggr)^{-3/2}\log^{-1}(2+\ve t)
\ee
 and denote by $\M$ the 3-dimensional vector
$(\M_1,\M_2,\M_3)$.
The goal of this section is to prove that  all these
majorants are bounded uniformly in $T$ for sufficiently small $\ve>0$.

%%%%%%%%%%%%%%%%%%%%%%%%%%%%%%%%%%%%%%%%%%%%%%%%%%%%%%%%%%%%%%%%%%%%%%%%%%%%%%%%%%%%
%%%%%%%%%%%%%%%%%%%%%%%%%%%%%%%%%%%%%%%%%%%%%%%%%%%%%%%%%%%%%%%%%%%%%%%%%%%%%%%%%%%%
\subsection{Bound for $g$}
\label{k1-est}
%%%%%%%%%%%%%%%%%%%%%%%%%%%%%%%%%%%%%%%%%%%%%%%%%%%%%%%%%%%%%%%%%%%%%%%%%%%%%%%%%%%%%
\begin{lemma}\la{k1-l}
  For the function $g(t)$ defined in (\ref{k}), the following bound holds
  \be\la{k1-1}
    \Vert g(t)\Vert_{E_{-\si}}\le c|z(0)|^2\frac 1{(1+t)^{3/2}}
    \le c\frac{\ve}{(1+t)^{3/2}},\quad\si>5/2
  \ee
\end{lemma}
\begin{proof}
  By (\ref{k}), we have
%  \be\la{k1-2}
    $g=-e^{At}k(0)$
 % \ee
  and
  $k(0)=a_{20}z^2(0)+a_{11}z(0)\ov z(0)+a_{02}\ov z^2(0)$
  with $a_{ij}$ defined in (\ref{a11}), (\ref{a20}).
  Therefore, Lemma \ref{Amu} and assumption (\ref{z0}) imply (\re{k1-1}).
\end{proof}

%%%%%%%%%%%%%%%%%%%%%%%%%%%%%%%%%%%%%%%%%%%%%%%%%%%%%%%%%%%%%%%%%%%%%%%%%%%%%%%%%%%%%%%%%%%

\subsection{Estimate of the remainders}
\label{maj-rem}
%%%%%%%%%%%%%%%%%%%%%%%%%%%%%%%%%%%%%%%%%%%%%%%%%%%%%%%%%%%%%%%%%%%%%%%%%%%%%%%%%%%%%
%%%%%%%%%%%%%%%%%%%%%%%%%%%%%%%%%%%%%%%%%%%%%%%%%%%%%%%%%%%%%%%%%%%%%%%%%%%%%%%%%%%%%%
\begin{lemma}\la{YR-est}
  The remainder $Y_R$ defined in (\ref{YR}) admits the estimate
  \be\la{YRest}
    |Y_R|={\cal R}(\ve^{1/2}\M)\frac{\ve^{5/2}}{(1+\ve t)^2\sqrt{\ve t}}
    \log(2+\ve t)(1+|\M|)^5
  \ee
\end{lemma}
\begin{proof}
  Using  the equality $f=k+g+h$ and estimate (\ref{YR}), we obtain
 $$
    |Y_R|={\cal R}_2(|z|+\Vert f\Vert_{L^{\infty}})|z|
    \Bigl[(|z|^2+\Vert g\Vert_{E_{-\si}}+\Vert h\Vert_{E_{-\si}})^2
    +|z|(\Vert g\Vert_{E_{-\si}}+\Vert h\Vert_{E_{-\si}})\Bigr]
   $$
  $$
   ={\cal R}(\ve^{1/2}\M)\Bigl(\frac{\ve}{1+\ve t}\Bigr)^{1/2}
    \M_1\Biggl[\Bigl(\frac{\ve}{1+\ve t}\M_1^2+\frac{\ve}{(1+ t)^{3/2}}
    +\Bigl(\frac{\ve}{1+\ve t}\Bigr)^{3/2}\log(2+\ve t)\!\M_3\Bigl)^2
  $$
  $$
    +\Bigl(\frac{\ve}{1+\ve t}\Bigr)^{1/2}\M_1\Bigl(\frac{\ve}{(1+ t)^{3/2}}
   +\Bigl(\frac{\ve}{1+\ve t}\Bigr)^{3/2}\log(2+\ve t)\M_3\Bigl)\Biggr]
  $$
  $$
   ={\cal R}(\ve^{1/2}\M)\frac{\ve^{5/2}}
    {(1+\ve t)^2\sqrt{\ve+\ve t}}\log(2+\ve t)(1+|\M|)^5
  $$
\end{proof}
%%%%%%%%%%%%%%%%%%%%%%%%%%%%%%%%%%%%%%%%%%%%%%%%%%%%%%%%%%%%%%%%%%%%%%%%%%%%%%%%%%%%
Now let us turn  to the remainders $F_R=F_{I}+F_{II}+F_{III}$,
$\tilde F_R=P^c{\cal N}_2[w,w]+F_R$, and
$H_R= F_R+H_I$ in equations (\ref{f-dot}) and (\ref{h-dot}) for $f$ and
$h$ respectively.
\begin{lemma}\la{HH}For $0<\nu<1/2$
  the remainder $F_R$ admits the  bound
  \be\la{fHR}
    \Vert F_R\Vert_{E_{5/2+\nu}}={\cal R}(\ve^{1/2}\M)
    \Bigl(\frac{\ve}{1+\ve t}\Bigr)^{3/2}\log(2+\ve t)
     \Bigl((\M_1+\M_2)(1+\M_1^2)+\ve^{1/2-\nu}(1+|\M|)^{12}\Bigr)
  \ee

\end{lemma}
%%%%%%%%%%%%%%%%%%%%%%%%%%%%%%%%%%%%%%%%%%%%%%%%%%%%%%%%%%%%%%%%%%%%%%%%%%%%%%%%%%%%%%%
\begin{proof}
{\it Step i)}
(\ref{FI}) and (\ref{fh}) imply for $\si=5/2+\nu$
$$
\Vert F_{I}+F_{II}\Vert_{E_{\si}}={\cal R}(|z|+\Vert f_1\Vert_{L^{\infty}})
(|z|^3+|z|\Vert f\Vert_{E_{-\si}}+\Vert f_1\Vert_{L^{\infty}}\Vert f\Vert_{E_{-\si}})
$$
$$
={\cal R}(|z|+\Vert f_1\Vert_{L^{\infty}})
\Big[|z|^3+|z|\Vert g\Vert_{E_{-\si}}+|z|\Vert h\Vert_{E_{-\si}}
+\Vert f_1\Vert_{L^{\infty}}(|z|^2+\Vert g\Vert_{E_{-\si}}+\Vert h\Vert_{E_{-\si}})\Big]
$$
$$
 ={\cal R}(\ve^{1/2}\M)\Biggl(\Bigl(\frac{\ve}{1+\ve t}\Bigr)^{3/2}\M_1^3
+\Big(\frac{\ve}{1+\ve t}\Bigr)^{1/2}\frac{\ve}{(1+t)^{3/2}}\M_1
+\Big(\frac{\ve}{1+\ve t}\Big)^2\log(2+\ve t)\M_1\M_3
$$
$$
+\Bigl(\frac{\ve}{1+\ve t}\Big)^{1/2}\log(2+\ve t)\M_2\Big[
\frac{\ve}{1+\ve t}\M_1^2+\frac{\ve}{(1+t)^{3/2}}
+\Bigl(\frac{\ve}{1+\ve t}\Big)^{3/2}\log(2+\ve t)\M_3\Big]\Biggl)
$$
which implies (\ref{fHR}) for $F_{I}+F_{II}$.
\\
{\it Step ii)}
Let us consider $\Vert F_{III}\Vert_{E_{\si}}$.
Bound (\ref{FIII}) implies
$$
\Vert F_{III}\Vert_{E_{5/2+\nu}}={\cal R}(|z|+\Vert f_1\Vert_{L^{\infty}})
(1+t)^{4+\nu}(|z|^{12}+\Vert f_1\Vert_{L^{\infty}}^{12})
$$
$$
={\cal R}(\ve^{1/2}\M)(1+t)^{4+\nu}\log^{12}(2+\ve t)
\Bigl(\frac{\ve}{1+\ve t}\Big)^{6}(M_1^{12}+M_2^{12})
$$
and then bound (\ref{fHR}) for $F_{III}$ follows.
%%%%%%%%%%%%%%%%%%%%%%%%%%%%%%%%%%%%%%%%%%%%%%%%%%%%%%%%%%%%%%%%
\begin{lemma}\la{HHW} For $0<\nu<1/2$
  the remainder $\tilde F_R$ admits the  bound
  \be\la{fHR2}
    \Vert \tilde F_{R}\Vert_{E_{5/2+\nu}\cap W}={\cal R}(\ve^{1/2}\M)
    \frac{\ve}{1+\ve t}\Big(M_1^2+\ve^{1/2}(1+|\M|)^{12}\Big)
  \ee
\end{lemma}
For $F_{I}$ and $F_{II}$ the bound follows from the estimate (\re{FI}).
Further, by (\ref{FIIIW})
$$
\Vert F_{III}\Vert_W={\cal R}(|z|+\Vert f_1\Vert_{L^{\infty}})
(|z|^{10}+\Vert f_1\Vert_{L^{\infty}}^{10})
={\cal R}(\ve^{1/2}\M)\Bigl(\frac{\ve}{1+\ve t}\Big)^{5}\log^{10}(2+\ve t)
(M_1^{10}+M_2^{10})
$$
which together with (\ref{fHR}) implies (\ref{fHR2}) for $F_{III}$.
Finally, (\ref{F2-est}) implies
$$
\Vert P^c{\cal N}_2[w,w]\Vert_{E_\si\cap W}={\cal R}(\ve^{1/2}\M)
\frac{\ve}{1+\ve t}\M_1^2
$$
and then (\ref{fHR2}) follows.
\end{proof}
%%%%%%%%%%%%%%%%%%%%%%%%%%%%%%%%%%%%%%%%%%%%%%%%%%%%%%%%%%%%%%%%%%%%%%%%%%%%%%
The  term $H_I$ is represented by (\ref{Am}) with $C_m$ estimated in (\ref{Cm-est1}).
For  $C_m$ we now obtain
\begin{lemma}\la{fAm}
For $m=0,~\pm 1$, the bounds hold
  \be\la{Am-final}
     \Vert C_{m}\Vert_{E_{\si}}= {\cal R}(\ve^{1/2}\M)
     \Bigl(\frac{\ve}{1+\ve t}\Bigr)^{3/2}\Big(\M_1^3+\ve^{1/2}(1+|\M|)^3\Big).
  \ee
\end{lemma}
\begin{proof}
  Estimate (\ref{Cm-est1}) implies
$$
\Vert C_{m}\Vert_{E_{\si}}={\cal R}(|z|+\Vert f\Vert_{E_{-\si}})
    |z|(|z|+\Vert g\Vert_{E_{-\si}}+\Vert h\Vert_{E_{-\si}})^2
$$
$$
   = {\cal R}(\ve^{1/2}\M)
    \Bigl(\frac{\ve}{1+\ve t}\Bigr)^{1/2}\M_1
   \Bigl(\Bigl(\frac{\ve}{1+\ve t}\Bigr)^{1/2}\M_1+\frac{\ve}{(1+t)^{3/2}}
     +\Bigl(\frac{\ve}{1+\ve t}\Bigr)^{3/2}\log(2+\ve t)\M_2\Bigr)^2
  $$
   which implies (\ref{Am-final}).
\end{proof}
%%%%%%%%%%%%%%%%%%%%%%%%%%%%%%%%%%%%%%%%%%%%%%%%%%%%%%%%%%%%%%%%%%%%%%%%%%%%%%%
%%%%%%%%%%%%%%%%%%%%%%%%%%%%%%%%%%%%%%%%%%%%%%%%%%%%%%%%%%%%%%%%%%%%%%%%%%%%%%%%%
%%%%%%%%%%%%%%%%%%%%%%%%%%%%%%%%%%%%%%%%%%%%%%%%%%%%%%%%%%%%%%%%%%%%%%%%%%%%%%%%%%%%
\subsection{Initial conditions}
\label{om-est}
%%%%%%%%%%%%%%%%%%%%%%%%%%%%%%%%%%%%%%%%%%%%%%%%%%%%%%%%%%%%%%%%%%%%%%%%%%%%%%%%%%%%%
We suppose the smallness of initial condition: \be\la{z0}
   |z(0)|\le\ve^{1/2}
\ee \be\la{h0}
  \Vert f(0)\Vert_{E_{\si}}=\Vert h(0)\Vert_{E_{\si}}\le \ve^{3/2}h_0
\ee
\be\la{f0}
\Vert f(0)\Vert_{E_\si\cap W}\le \ve^{1/2}f_0
\ee
where $h_0$, $f_0$ are some fixed constant, and $\ve>0$
is sufficiently small accordingly (\ref{close}). Equation
(\ref{z-rel}) implies $|z_1|^2\le |z|^2+{\cal R}(|z|)|z|^3$.
Therefore \be\la{y0}
 y_0= y(0)=|z_1(0)|^2\le\ve+C(|z(0)|)\ve^{3/2}
\ee
%%%%%%%%%%%%%%%%%%%%%%%%%%%%%%%%%%%%%%%%%%%%%%%%%%%%%%%%%%%%%%%%%%%%%%%%%%%%%%%%%%%%
%%%%%%%%%%%%%%%%%%%%%%%%%%%%%%%%%%%%%%%%%%%%%%%%%%%%%%%%%%%%%%%%%%%%%%%%%%%%%%%%%%%%
\subsection{Estimates via majorants}
\label{can-sys}
%%%%%%%%%%%%%%%%%%%%%%%%%%%%%%%%%%%%%%%%%%%%%%%%%%%%%%%%%%%%%%%%%%%%%%%%%%%%%%%%%%%%%
This section is devoted to the study the system (\ref{f-dot}), (\ref{h-dot}), (\ref{y-s})
under  assumptions (\ref{h0}), (\ref{f0}), (\ref{y0}) on initial data
and the estimates (\ref{YRest}), (\ref{fHR}), (\ref{fHR2}), (\ref{Am-final})
of the remainders.

First we consider equation (\ref{y-s}) for $y$ which is of Ricatti type.
\begin{lemma}\la{my-sol}(\cite {BS}, Proposition 5.6)
  The solution to (\ref{y-s}) with an initial condition and a remainder
  satisfying (\ref{y0}) and (\ref{YRest}) respectively admits the bound:
  \be\la{Ricatti}
     |y-\frac{y_0}{1+2 y_0\Im Kt}|
    \le {\cal R}(\ve^{1/2}\M)\frac{\ve^{5/2}}{(1+\ve t)^2\sqrt{\ve t}}
    \log(2+\ve t)(1+|\M|)^5.
  \ee
\end{lemma}
%%%%%%%%%%%%%%%%%%%%%%%%%%%%%%%%%%%%%%%%%%%%%%%%%%%%%%%%%%%%%%%%%%%%%%%%%%%%%%%%%%%%%%%%%
\begin{cor}\label{M1-est}
  The majorant $\M_1$ satisfies
  \be\la{M1M}
    \M_1^2={\cal R}(\ve^{1/2}\M)\Big(1+\ve^{1/2}(1+|\M|)^5\Big)
  \ee
\end{cor}
\begin{proof}
  Bounds (\ref{y0}) and (\ref{Ricatti})  imply
$$
    y\le {\cal R}(\ve^{1/2}\M)\Bigg[\frac{\ve}{1+\ve t}
    +\Big(\frac{\ve}{1+\ve t}\Big)^{3/2}\log(2+\ve t)(1+|\M|)^5\Bigg].
  $$
  Using that $|z|^2\le y+ {\cal R}(|z|)|z|^3$, we get
  $$
    |z|^2\le {\cal R}(\ve^{1/2}\M)\Bigg[\frac{\ve}{1+\ve t}+\Big(\frac{\ve}{1+\ve t}\Big)
    ^{3/2}\log(2+\ve t)(1+|\M|)^5+\Big(\frac{\ve}{1+\ve t}\Big)^{3/2}\M_1^3\Bigg]
  $$
  Hence, (\ref{M1M}) follows.
\end{proof}
%%%%%%%%%%%%%%%%%%%%%%%%%%%%%%%%%%%%%%%%%%%%%%%%%%%%%%%%%%%%%%%%%%%%%%%%%%%%%%%%%%%%
Second we consider equation (\ref{f-dot}) for $f$.
\begin{lemma}\la{mf-s}
The solution to (\ref{f-dot}) admits the bound
\be\la{f-est}
    \Vert f_1\Vert_{L^{\infty}}\le C\Bigl(\frac{\ve}
    {1+\ve t}\Bigr)^{1/2}\log(2+\ve t)\Big(f_{0}
    +{\cal R}(\ve^{1/2}\M)(\M_1^2+\ve^{1/2}(1+|\M|)^{12}\Big).
\ee
\end{lemma}
\begin{proof}
The solution $f(x,t)$ to (\ref{f-dot}) is expressed as
 $$
 f(t)=e^{At}f(0) +\int\limits_0^te^{A(t-\tau)}\tilde F_R(\tau)d\tau
 $$
Using the the bounds (\ref{t-as3}) and the estimates (\ref{fHR2}), (\ref{f0}),
we obtain
$$
  \Vert f_1\Vert_{L^{\infty}}\le \frac{C}{(1+t)^{1/2}}
  \Vert f(0)\Vert_{E_\si\cap W}+\int\limits_0^t\frac{C}
  {(1+(t-\tau))^{1/2}}\Vert \tilde F_R(\tau)\Vert_{E_\si\cap W}d\tau
$$
$$
  \le C \Bigg[f_0\Big(\frac{\ve}{1+t}\Big)^{1/2}
  +{\cal R}(\ve^{1/2}\M)(\M_1^2+\ve^{1/2}(1+|\M|)^{12}\int\limits_0^t\frac {d\tau}
  {(t-\tau)^{1/2}}\frac{\ve}{1+\ve\tau}d\tau\Bigg]
$$
$$
  \le C\Bigl(\frac{\ve}{1+\ve t}\Bigr)^{1/2}\log(2+\ve t)\Big(f_0+
  {\cal R}(\ve^{1/2}\M)(\M_1^2+\ve^{1/2}(1+|\M|)^{12}\Big)
$$
\end{proof}
\begin{cor}
\be\la{M2M}
    \M_2={\cal R}(\ve^{1/2}\M)\Big(M_1^2+\ve^{1/2}(1+|\M|)^{12}\Big).
  \ee
\end{cor}
%%%%%%%%%%%%%%%%%%%%%%%%%%%%%%%%%%%%%%%%%%%%%%%%%%%%%%%%%%%%%%%%%%%%%%%%%%%%
Finally  we consider equation (\ref{h-dot}) for $h$.
\begin{lemma}\la{mh-sol}
The solution to (\ref{h-dot}) admits the bound
\be\la{h-est}
    \Vert h\Vert_{E_{-\si}}\le C\Bigl(\frac{\ve}
    {1+\ve t}\Bigr)^{3/2}\!\log(2+\ve t)\Bigg(h_{0}
    +{\cal R}(\ve^{1/2}\M)\Big[(1+\M_1^2+\M_2)
    (1+\M_1^2)+\ve^{1/2-\nu}(1+|\M|)^{12}\Big]\Bigg)
\ee
\end{lemma}
\begin{proof}
 The solution  $h(x,t)$ to (\ref{h-dot}) is expressed as
$$
     h=e^{At}h(0)+\int\limits_0^te^{A(t-\tau)}H_R(\tau)d\tau
$$
Using the bounds (\ref{h0}), (\ref{fHR}), (\ref{Am-final}), Proposition
\ref{1d} and Corollary  \ref{Amu}, we get
$$
  \Vert h\Vert_{E_{-\si}}\le \frac C{(1+t)^{3/2}}
  \Vert h(0)\Vert_{E_\si}+\int\limits_0^t\frac C
  {(1+(t-\tau))^{3/2}}\Big(\Vert F_R(\tau)\Vert_{E_\si}+
  \sum\limits_m\Vert C_m(\tau)\Vert_{E_\si}\Big)d\tau
$$
$$
  \le C\Bigg[h_0\Big(\frac{\ve}{1\!+\!t}\Big)^{\!\fr 32}\!
  +{\cal R}(\ve^{\!\fr 12}\M)\Big[(\M_1+\M_2)
  (1+\M_1^2)+\ve^{\!\fr 12-\nu}(1+|\M|)^{12}\Big]\!\int\limits_0^t
  \frac {\log(2\!+\!\ve t)~d\tau}{(1\!+\!(t\!-\!\tau))^{\!\fr 32}}
  \Big(\frac{\ve}{1\!+\!\ve\tau}\Big)^{\!\fr 32}
$$
$$
+\sum_m {\cal R}(\ve^{1/2}\M)\Big(\M_1^3+\ve^{1/2}(1+|\M|)^{3}\Big)
\int\limits_0^t\frac {d\tau}{(1+(t-\tau))^{3/2}}
  \Big(\frac{\ve}{1+\ve\tau}\Big)^{3/2}\Bigg]
$$
which implies (\ref{h-est}).
\end{proof}
\begin{cor}
\be\la{M3M}
    \M_3={\cal R}(\ve^{1/2}\M)\Bigg[(1+\M_1^2+\M_2)
    (1+\M_1^2)+\ve^{1/2-\nu}(1+|\M|)^{12}\Big].
  \ee
\end{cor}
%%%%%%%%%%%%%%%%%%%%%%%%%%%%%%%%%%%%%%%%%%%%%%%%%%%%%%%%%%%%%%%%%%%%%%%%%%%%%%%%%%%%
%%%%%%%%%%%%%%%%%%%%%%%%%%%%%%%%%%%%%%%%%%%%%%%%%%%%%%%%%%%%%%%%%%%%%%%%%%%%%%%%%%%%
\subsection{Uniform bounds for majorants}
\label{M-bounds}
%%%%%%%%%%%%%%%%%%%%%%%%%%%%%%%%%%%%%%%%%%%%%%%%%%%%%%%%%%%%%%%%%%%%%%%%%%%%%%%%%%%%%
The aim of this section is to prove that if $\ve$ is sufficiently small, all the $\M_i$
are bounded uniformly in $T$ and $\ve$.\\
\begin{lemma}\la{Mb}
  For $\ve$ sufficiently small, there exists a constant $M$ independent of $T$ and $\ve$,
  such that,
  \be\la{M-est}
    |\M(T)|\le M.
  \ee
\end{lemma}
%%%%%%%%%%%%%%%%%%%%%%%%%%%%%%%%%%%%%%%%%%%%%%%%%%%%%%%%%%%%%%%%%%%%%%%%%%%%%%%%%%%%%%%
\begin{proof}
  Combining the inequalities (\ref{M1M}), (\ref{M2M}), and (\ref{M3M})
  for the $\M_i$, we obtain the inequality
  $$
   \M^2\le{\cal R}(\ve^{1/2}\M)\Big[(1+\M_1^2+\M_2)^4
   +\ve^{1/2-\nu}(1+|\M|)^{24}\Big]
  $$
  Replacing $\M_1^2$ and $\M_2$ in the right-hand by its bound (\ref{M1M})
  and (\ref{M2M}), we obtain
  $$
   \M^2\le{\cal R}(\ve^{1/2}\M)(1+\ve^{1/2-\nu}F(\M))
  $$
  where $F(\M)$ is an appropriate function. This inequality implies
  that $\M$ is bounded uniformly in $\ve$ since $\M(0)$ is small
  and $\M(t)$ is continuous.
\end{proof}
%%%%%%%%%%%%%%%%%%%%%%%%%%%%%%%%%%%%%%%%%%%%%%%%%%%%%%%%%%%%%%%%%%%%%%%%%%%%%%%%
%%%%%%%%%%%%%%%%%%%%%%%%%%%%%%%%%%%%%%%%%%%%%%%%%%%%%%%%%%%%%%%%%%%%%%%%%%%%%%%%%%%%
%\subsection{Decay of the solution}
%\label{lt-as}
%%%%%%%%%%%%%%%%%%%%%%%%%%%%%%%%%%%%%%%%%%%%%%%%%%%%%%%%%%%%%%%%%%%%%%%%%%%%%%%%%%%%%
\begin{cor}\la{bounds}
  The following estimates hold for all $t>0$, $\si>5/2$
  \be\la{est1}
    |z(t)|\le M\Big(\frac{\ve}{1+\ve t}\Big)^{1/2},
  \ee
  \be\la{est2}
    \Vert f_1\Vert_{L^{\infty}}\le M\Big(\frac{\ve}{1+\ve t}\Big)^{1/2}\log(1+\ve t),
 \ee
  \be\la{est3}
    \Vert h\Vert_{E_{-\si}}\le M\Big(\frac{\ve}{1+\ve t}\Big)^{3/2}\log(1+\ve t).
  \ee
  \be\la{est4}
    \Vert f\Vert_{E_{-\si}}\le M\Big(\frac{\ve}{1+\ve t}\Big).
  \ee
\end{cor}
%%%%%%%%%%%%%%%%%%%%%%%%%%%%%%%%%%%%%%%%%%%%%%%%%%%%%%%%%%%%%%%%%%%%%%%%%%%%%%%%%%%%
Thus we have proved the following result:
\begin{theorem}\la{t-main}
  Let the  conditions of Theorem \ref{main} hold.
  Then\\
  {\it i)} for $\ve$ small enough, one can write the solution of
  (\ref{Eq}) in the form
  \be\la{solform}
     Y(x,t)=
     s(x)+(z(t)+\ov z(t))u+f(x,t),
  \ee
  {\it ii)}  In addition, for all $t>0$, there exists a constant $M>0$ such that
  \be\la{est00}
   |z(t)|\le M\Big(\frac{\ve}{1+\ve t}\Big)^{1/2}, \quad
    \Vert f\Vert_{E_{-\si}}\le M\Big(\frac{\ve}{1+\ve t}\Big),\quad\si>5/2.
  \ee
\end{theorem}
%%%%%%%%%%%%%%%%%%%%%%%%%%%%%%%%%%%%%%%%%%%%%%%%%%%%%%%%%%%%%%%%%%%%%%%%%%%%%%%%%%%%%%%%%
%%%%%%%%%%%%%%%%%%%%%%%%%%%%%%%%%%%%%%%%%%%%%%%%%%%%%%%%%%%%%%%%%%%%%%%%%%%%%%%%%%%%%%%%%%
\section{Soliton asymptotics}
\label{solas-sec}
%%%%%%%%%%%%%%%%%%%%%%%%%%%%%%%%%%%%%%%%%%%%%%%%%%%%%%%%%%%%%%%%%%%%%%%%%%%%%%%%%%%
\subsection{Long time behavior of $z(t)$}
\label{z-beh}
%%%%%%%%%%%%%%%%%%%%%%%%%%%%%%%%%%%%%%%%%%%%%%%%%%%%%%%%%%%%%%%%%%%%%%%%%%%%%%%%%%%%%%%%%%%%
%%%%%%%%%%%%%%%%%%%%%%%%%%%%%%%%%%%%%%%%%%%%%%%%%%%%%%%%%%%%%%%%%%%%%%%%%%%%%%%%%%%%%%%%%
We start with equation (\ref{z4}) for $z_1$ that we rewrite
$$
  \dot z_1=i\mu z_1+iK|z_1|^2z_1+\widehat Z_R
$$
By (\ref{tZRest}) $\widehat Z_R$
satisfies the estimate
$$
 \widehat Z_R
  ={\cal R}(|z|+\Vert f_1\Vert_{L^\infty})\Bigl[(|z|^2+\Vert f\Vert_{E_{-\si}})^2
  +|z|\Vert g\Vert_{E_{-\si}}+|z|\Vert h_1\Vert_{E_{-\si}}\Bigr]
$$
$$
  ={\cal R}(\ve^{1/2}M)\frac{\ve^2\log(2+\ve t)}{(1+\ve t)^{3/2}\sqrt{\ve t}}
   (1+M^4)\le\frac{C\ve^2\log(2+\ve t)}{(1+\ve t)^{3/2}\sqrt{\ve t}}
$$
On the other hand, we have, from (\ref{y0}) and (\ref{Ricatti}),
$$
\Big|y-\frac{y_0}{1+2\Im Ky_0t}\Big|
\le C\Big(\frac{\ve}{1+\ve t}\Big)^{3/2}\log(2+\ve t)
$$
with $|y_0-\ve|\le C\ve^{3/2}$.
With estimate  (\ref{est1}) for $|z|$ and obviously the same one for
$|z_1|$, we have
\be\la{z7}
\dot z_1= i\mu z_1+iK\frac{y_0}{1+2\Im Ky_0t}z_1+Z_1
\ee
with
$$
|Z_1|\le\frac{C\ve^2\log(2+\ve t)}{(1+\ve t)^{3/2}\sqrt{\ve t}}
$$
Since $y_0=\ve+{\cal O}(\ve^{3/2})$, we have that the coefficient
$2\Im Ky_0=k\ve$. We also denote $\rho=\frac{\Re K}{\Im K}$.
The solution $z_1$ of (\ref{z7}) is written in the form
$$
 z_1=\frac{e^{i\mu t}}{(1+k\ve t)^{1/2-i\rho}}
 \Big[z_1(0)+\int\limits_0^t e^{-i\mu s}(1+k\ve s)^{1/2-i\rho}Z_1(s)ds\Big]
  =z_{L^{\infty}}\frac{e^{i\mu t}}{(1+k\ve t)^{1/2-i\rho}}+z_R
$$
where
$$
 z_{L^{\infty}}= z_1(0)+\int\limits_0^{L^{\infty}}e^{-\mu s}(1+k\ve s)^{1/2-i\rho}Z_1(s)ds
$$
and
$$
 z_R= -\int\limits_t^{L^{\infty}}e^{i\mu t}
 \Big(\frac{1+k\ve s}{1+k\ve t}\Big)^{1/2-i\rho}Z_1(s)ds.
$$
From the bound (\ref{z7}) on $Z_1$ it follows that
$$
 |z_R|\le\frac{C\ve\log(2+\ve t)}{(1+\ve t)}.
$$
Therefore $z_1(t)$  satisfies the estimate
\be\la{zt}
z_1(t)= z_{L^{\infty}}\frac{e^{i\mu t}}{(1+k\ve t)^{1/2-i\rho}}
 +{\cal O}\Big(\frac{\ve}{1+\ve t}\log(2+\ve t)\Big)
\ee
Here $z_\infty=z_1(0)+{\cal O}(\ve)$, $z=z_1+{\cal O}(\ds\frac{\ve}{1+\ve t})$,
and $|z(0)|=\ve^{1/2}$. Thus $|z_\infty|=\ve^{1/2}+{\cal O}(\ve)$.
Hence, the function $z(t)$ can be  estimated as
\be\la{zfin}
z(t)=z_\infty\frac{e^{i\mu t}}{(1+k\ve t)^{1/2-i\rho}}
 +{\cal O}\Big(\frac{\ve}{1+\ve t}\log(2+\ve t)\Big)
\ee
%%%%%%%%%%%%%%%%%%%%%%%%%%%%%%%%%%%%%%%%%%%%%%%%%%%%%%%%%%%%%%%%%%%%%%%%%%%%%%%%%%%%%%%%%
%%%%%%%%%%%%%%%%%%%%%%%%%%%%%%%%%%%%%%%%%%%%%%%%%%%%%%%%%%%%%%%%%%%%%%%%%%%%%%%%%%%%%%%%%%
\subsection{Proof of soliton asymptotics}

%%%%%%%%%%%%%%%%%%%%%%%%%%%%%%%%%%%%%%%%%%%%%%%%%%%%%%%%%%%%%%%%%%%%%%%%%%%%%%%%%%%%%%%%%%%%
Here we prove our main Theorem \ref{main}.
We have obtained the solution $Y(x,y)$ to (\ref{Eq}) in the form
\be\la{dY}
Y=S+w+f,
\ee
We include $w$ into the remainder $r_{\pm}$ from (\ref{S})
since $z(t)\sim t^{-1/2}$ by (\ref{zfin}).
It remains to extract the dispersive wave
$W(t)\Phi_{\pm}$ from the term $f$.
\subsubsection{The asymptotic completeness}
Let us rewrite equation (\ref{f}) as
\be\la{sys}
\left\{\ba{rcl}\dot f_1&=&f_2+Q_1
\\\dot f_2&=&f_1''-m^2 f_1+Q_2\ea\right.
\ee
where
\beqn\nonumber
Q_1&=&(P^c{\cal N})_1=-(P^d{\cal N})_1=-\fr 1{i\de}\langle N,u_1\rangle u_1+
\fr 1{i\de}\langle N,u_1\rangle u_1=0\\
\nonumber
Q_2&=&(P^c{\cal N})_2=(P^c{\cal N}_2[w,w])_2+(F_R)_2-Vf_1
\eeqn
 by (\ref{PdX}) and (\ref{uu}).
The equations imply
the asymptotics of type
(\ref{rm}),
\beqn\nonumber f(t)\!\!\!&=&\!\!\!W_0(t)f(0)+\int\limits_0^t
W_0(t-\tau)Q(\tau)d\tau= W_0(t)\Big(f(0)+
\int\limits_0^{\infty} W_0(-\tau)Q(\tau)d\tau\Big)\\
\la{f-sol}
\!\!\!&-&\!\!\!
\int\limits^{\infty}_t W_0(t-\tau)Q(\tau)d\tau=W_0(t)\phi_{+}+r_{+}(t),
\eeqn
if all the integrals converge.
Here $W_0(t)$ is the dynamical group of
the free Klein-Gordon equation,
and $Q(t):=(0,Q_2(t))$.
To complete the proof of
(\ref{rm}),
it remains to
prove the following proposition
\bp\la{prem} The bounds hold
\be\la{phi-r}
\Vert r_{+}(t)\Vert_{E}=\cO(t^{-1/3}),~~~~~~t\to\infty
\ee
\ep
\begin{proof}
To check (\re{phi-r}), we should
obtain an appropriate decay for the function  $Q_2(t)$.
\medskip\\
{\it Step i)}
First, according to (\ref{FR}), (\ref{FI}), (\ref{FIII}),
(\ref{est1}), (\ref{est2}), and (\ref{est4}), we have
\be\la{QQ}
\Vert (F_R)_2\Vert_{L^2}={\cal O}(t^{-3/2}\log t)
\ee
By (\ref{PdX}), (\ref{N2uu}), and (\ref{Zij})
$$
(P^c{\cal N}_2[w,w])_2=N_2[w,w]-(P^d{\cal N}_2[w,w])_2
=(z^2+2z\ov z+\ov z^2)\Big(N_2[u,u]-2i\mu u_1Z_2\Big)
$$
Hence,
 (\ref{h-dec}) and (\ref{k}) imply that
\be\la{Q2}
Q_2=q_{20}z^2+2q_{11}z\ov z+q_{02}\ov z^2+Q_{2R}
\ee
with
\be\la{pp}
q_{ij}=N_2[u_1,u_1]-2iZ_2\mu u_1 -Va_{ij,1},\quad~~~~~
Q_{2R}=(F_R)_2 - V(f_1-k_1)
\ee
where $a_{ij,1}$ and $k_1$ are the first components of vector-functions
$a_{ij}$ and $k$ from (\ref{k}).
By (\ref{V-decay}), (\ref{h-dec}), (\ref{k1-1}) and   (\ref{est3})
$$
\Vert V(f_1-k_1)\Vert_{L^2}={\cal O}(t^{-3/2}\log t),\quad t\to\infty,
$$
Hence, the last bound and (\ref{QQ}) imply that
\be\la{VQ2}
\Vert Q_{2R}\Vert_{L^2}={\cal O}(t^{-3/2}\log t),\quad t\to\infty.
\ee
Therefore, the term $Q_{2R}$ give the contribution
of  order $\cO(t^{-1/2}\log t)$ to $r_+(t)$.
\medskip\\
{\it Step ii)}
It remains
to estimate  the contribution
to $r_+(t)$
of the quadratic terms
$q_{ij}z^i\ov z^j$
from (\re{Q2}).
Let us note that  the functions $q_{ij}(x)$ are smooth
with exponential decay at infinity
similarly to the functions $u_1(x)$ and $V(x)$ since
$a_{ij}\in H^s_{-\si}$ with any $s>0$ by Lemma
\re{h-trans}.

On the other hand, the time decay of the functions
$z^i(t)\ov z^j(t)$ is very slow like $\cO(t^{-1})$. Therefore, the
contribution of the term $q_{ij}z^i\ov z^j$ to $r_+(t)$
is the integral of type (\re{f-sol})
which does not converge absolutely.
Fortunately, we may define the integral as
\be
\ds\int_t^\infty W(t-\tau)q_{ij}(\tau)z^i\ov z^jd\tau
:=\lim_{T\to\infty}
\ds\int_t^T W(t-\tau)q_{ij}(\tau)z^i\ov z^jd\tau
\ee
We prove below the convergence of the integrals with the values
in $E$ and the decay rate $\cO(t^{-1/3})$.

First we estimate the contribution
of the  term $q_{11}(x)z\ov z$.
Note that
(\ref{zfin}) implies the asymptotics
$z\ov z\sim (1+k\ve t)^{-1}$.
%%%%%%%%%%%%%%%%%%%%%%%%%%%%%%%%%%%%%%%%%%%%%%%%%%%%%%%%%%%%%%%%%%%%%%%%%%%%%%%%%%
\begin{lemma}\la{l1}
Let $q(x)\in L^2(\R)$. Then
\be\label{tQ}
I(t):=\Big\Vert\int\limits_t^{\infty}\! W_0(-\tau)
\left(\!\ba{c} 0 \\ q\ea\!\right)\fr{d\tau}{1+\tau}\Big\Vert_{E}
={\cal O}(t^{-1}),\quad t\to\infty.
\ee
\end{lemma}
\begin{proof}
Denote $\om=\om(\xi)=\sqrt{\xi^2+m^2}$. Then
\be\la{It}
I(t)=\Big\Vert\int\limits_t^{\infty}
\left(\ba{c}
-\sin\om\tau~\hat q(\xi) \\ -\cos\om\tau~\hat q(\xi)
\ea\right)\frac{d\tau}{1+\tau}\Big\Vert_{L^2\oplus L^2}\le \frac C{1+t}
\Vert \hat q(\xi)/\om(\xi)\Vert_{L^2}
\ee
since the partial integration implies that
\be\la{cs}
\Big|\int\limits_t^{\infty}\fr{e^{i\om\tau}}{1+\tau}~d\tau\Big|
=\Big|\int\limits_t^{\infty}\fr{de^{i\om\tau}}{i\om(1+\tau)}~d\tau\Big|
\le\Big|\fr{e^{i\om\tau}}{\om(1+t)}\Big|
+\Big|\int\limits_t^{\infty}\fr{e^{i\om\tau}}{\om(1+\tau)^2}~d\tau\Big|
\le\frac{C}{\om(1+t)}
\ee
\end{proof}
%%%%%%%%%%%%%%%%%%%%%%%%%%%%%%%%%%%%%%%%%%%%%%%%%%%%%%%%%%%%%%%%%%%%%%%%%%%%%%%%%%%%%%%
Next we estimate the contribution from the terms with
$q_{20}(x)z^2$ and
$q_{02}(x)\ov z^2$ to (\ref{f-sol}) (cf. \ci[Proposition 6.5]{BS}).
Now (\ref{zfin}) implies the asymptotics
$z^2\sim e^{2i\mu\tau}/(1+k\ve t)^{1-2i\rho}$ and
$\ov z^2\sim e^{-2i\mu\tau}/(1+k\ve t)^{1+2i\rho}$.

\begin{lemma}\la{l2}
Let $q(x)\in L^2(\R)\cap L^1(\R)$.
Then
\be\label{dtQm}
\Big\Vert\int\limits_t^{\infty} W_0(-\tau)
\left(\ba{c} 0 \\ q\ea\right)\frac{e^{\pm 2i\mu\tau}d\tau}
{(1+\tau)^{1\mp 2i\rho}}
\Big\Vert_{E}={\cal O}(t^{-1/3}),\quad t\to\infty.
\ee
\end{lemma}
%%%%%%%%%%%%%%%%%%%%%%%%%%%%%%%%%%%%%%%%%%%%%%%%%%%%%%%%%%%%%%%%
%%%%%%%%%%%%%%%%%%%%%%%%
\begin{proof} We consider for example
the integral with
$e^{- 2i\mu\tau}$
and
omit for simplicity the factor $(1+t)^{2i\rho}$ since with
the factor the proof
is similar. Let us represent $\sin\om\tau$ and $\cos\om\tau$
as linear combination
of $e^{i\om\tau}$ and  $e^{-i\om\tau}$.
The contribution from the ``nonresonant''
terms with the $e^{-i\om\tau}$ in (\ref{dtQm}) is
${\cal O}(t^{-1})$ similarly
to (\ref{It})-(\ref{cs}). It remains to prove that
\be\la{Pt}
I(t)=\Big\Vert\int\limits_t^{\infty}\fr{e^{i(\om-2\mu)\tau}\hat
q(\xi)~d\tau}
{1+\tau}\Big\Vert_{L^2}={\cal O}(t^{-1/3})
\ee
For the fixed $\beta>0$ let us denote
$$
\chi_{\tau}(\xi)=\left\{\ba{rcl}1,~~ |\om(\xi)-2\mu|\le 1/\tau^{\beta}
\\ 0,~~|\om(\xi)-2\mu|> 1/\tau^{\beta} \ea\right.
$$
 Then
$$
I(t)\le\Big\Vert\int\limits_t^{\infty}
\fr{e^{i(2\om-\mu)\tau}\chi_{\tau}(\xi)\hat q(\xi)~d\tau}
{1+\tau}\Big\Vert_{L^2}
+\Big\Vert\int\limits_t^{\infty}
\fr{e^{i(2\om-\mu)\tau}(1-\chi_{\tau}(\xi))\hat q(\xi)~d\tau}
{1+\tau}\Big\Vert_{L^2}
=I_1(t)+I_2(t)
$$
Since $\hat q(\xi)$ is bounded function,
and $\Vert\chi_\tau\Vert^2\le 1/\tau^\beta$, we have
$$
I_1(t)\le\fr{C\Vert\hat q\Vert_{L^\infty}}{(1+t)^{\beta/2}}
$$
On the other hand,
the partial integration implies that
$$
I_2(t)=\Big\Vert\int\limits_t^{\infty}
\fr{(1-\chi_{\tau}(\xi))\hat q(\xi)~de^{i(2\om-\mu)\tau}}
{(2\om-\mu)(1+\tau)}\Big\Vert_{L^2}\le
\fr{Ct^{\beta}}{1+t}\Vert\hat q\Vert_{L^2}
+C\int\limits_t^{\infty}\frac{\tau^{\beta}d\tau}
{(1+\tau)^2}\Vert\hat q\Vert_{L^2}\le
\fr{C\Vert\hat q\Vert_{L^2}}{(1+t)^{1-\beta}}
$$
%$$
%+\Big\Vert\int\limits_t^{\infty}
%\fr{e^{i\tau(2\om-\mu)}\pa_{\tau}(1-\chi_{\tau}(\xi))\hat q(\zeta)~d\tau}
%{(2\om-\mu)(1+\tau)}\Big\Vert_{L^2}
%$$
Equating $\ds\fr{\beta}2=1-\beta$, we get $\beta=\ds\fr 23$.
\end{proof}

Now Proposition \re{prem} is proved.
\end{proof}
%%%%%%%%%%%%%%%%%%%%%%%%%%%%%%%%%%%%%%%%%%%%%%%%%%%%%%%%%%%%%%%%%%%%%%%%%%%%%%

\protect\renewcommand{\thesection}{\Alph{section}}
\protect\renewcommand{\theequation}{\thesection. \arabic{equation}}
\protect\renewcommand{\thesubsection}{\thesection. \arabic{subsection}}
\protect\renewcommand{\thetheorem}{\Alph{section}.\arabic{theorem}}
%%%%%%%%%%%%%%%%%%%%%%%%%%%%%%%%%%%%%%%%%%%%%%%%%%%%%%%%%%%%%%%%%%%%%%%%%%%%%%%%%%%%%%%%%%%%
\section{Virial type estimates}
\setcounter{equation}{0}
\label{wes}
%%%%%%%%%%%%%%%%%%%%%%%%%%%%%%%%%%%%%%%%%%%%%%%%%%%%%%%%%%%%%%%%%%%%%%%%%%%%%%%%%%%%%%%
We prove the weighted estimate (\re{jv1}).
Let us recall that we split the solution
$Y(t)=(\psi(\cdot,t),\pi(\cdot,t))=S+X(t)$,
 and denote $X(t)=(\Psi(t),\Pi(t))$,
$(\Psi_0,\Pi_0):=(\Psi(0),\Pi(0))$.
Finally, our basic  condition (\ref{close}) implies that for some $\nu>0$.
\be\la{inn}
\Vert X_0\Vert_{E_{5/2+\nu}}\le d_0<\infty
\ee
\begin{pro}\la{ee2}
Let the potential $U$ satisfy conditions {\bf U1}, and
$\Psi_0$ satisfy (\re{inn}).
 Then the bounds hold
\be\la{solt}
\Vert\Psi(t)\Vert_{L^2_{5/2+\nu}}\le C(d_0)(1+t)^{4+\nu},
~~~~~~~~t>0
\ee
\end{pro}
We will deduce the proposition from the following two lemmas.
The first lemma is well known.
Denote
$$
e(x,t)=\ds\fr{|\pi(x,t)|^2}2+\ds\fr{|\psi'(x,t)|^2}2+ U(\psi(x,t)).
$$
%%%%%%%%%%%%%%%%%%%%%%%%%%%%%%%%%%%%%%%%%%%%%%%%%%%%%%%%%%%%%%%%%%%%%%%%%%%%
\begin{lemma}\la{ee}
For the solution $\psi(x,t)$ of Klein-Gordon equation (\ref{e})
the local energy estimate holds
\be\la{enes}
 \int\limits_{a}^{b}e(x,t)~dx
 \le\int\limits_{a-t}^{b+t}e(x,0)~dx,\quad a<b,\quad t>0.
\ee
\end{lemma}
\begin{proof}
The estimate follows by standard arguments: multiplication
of the equation (\re{e})
by $\dot\psi(x,t)$ and integration over the trapezium $ABCD$, where
$A=(a-t,0)$, $B=(a,t)$, $C=(b,t)$, $D=(b+t,0)$.
Then (\ref{enes}) is obtained after partial integration
using that $U(\psi)\ge 0$.
\end{proof}
%%%%%%%%%%%%%%%%%%%%%%%%%%%%%%%%%%%%%%%%%%%%%%%%%%%%%%%%%%%%%%%%%%%%%%%%%%%

\begin{lemma}\la{ee1} For any $\si\ge 0$
\be\la{wees}
\int(1+|x|^{\si})e(x,t)dx \le C(\si)(1+t)^{\si+1}\int(1+|x|^{\si})e(x,0)dx.
\ee
\end{lemma}
\begin{proof}
By (\ref{enes})
$$
\int (1+|x|^{\si})\Big(\int\limits_{x-1}^{x}e(y,t)dy\Big)dx
\le\int (1+|x|^{\si})\Big(\int\limits_{x-1-t}^{x+t}e(y,0)dy\Big)dx.
$$
Hence,
\be\la{wees1}
\int e(y,t)\Big(\int\limits_{y}^{y+1}(1+|x|^{\si})dx\Big)dy
\le\int e(y,0)\Big(\int\limits_{y-t}^{y+t+1}(1+|x|^{\si})dx\Big)dy.
\ee
Obviously,
\be\la{y7}
\int\limits_{y}^{y+1}(1+|x|^\si)dx
\ge c(\si)(1+|y|^\si)
\ee
with some  $c(\sigma)>0$. On the other hand,
\be\la{yt7}
\int\limits_{y-t}^{y+t+1}(1+|x|^\si)dx\le (2t+1)(1+t+|y|)^{\si}
\le C(\si)(1+t)^{\si+1}(1+|y|^{\si})
\ee
since $\si\ge 0$.
Now the bound (\ref{wees}) follows
from (\ref{wees1})-(\ref{yt7}) .
\end{proof}
%%%%%%%%%%%%%%%%%%%%%%%%%%%%%%%%%%%%%%%%%%%%%%%%%%%%%%%%%%%%%%%%%%%%%%%%%%%%%%%
%%%%%%%%%%%%%%%%%%%%%%%%%%%%%%%%%%%%%%%%%%%%%%%%%%%%%%%%%%%%%%%%%%%%%%%%%%%%%%
~\\
{\bf Proof of Proposition \ref{ee2}}
First we verify that
\be\la{U-int}
U_0:=\int(1+|x|^{5+2\nu})U(\psi_0(x))dx
<\infty ,\quad\psi_0(x)=\psi(x,0)
\ee
Indeed, $\psi_0(x)=s(x)+\Psi_0(x)$ is bounded since
$\Psi_0\in H^1(\R)$.
Hence
{\bf U1} implies for
$$
|U(\psi_0(x))|\le C(d_0)(\psi_0(x)\pm a)^2\le C(d_0)\Big((s(x)\pm a)^2
+\Psi_0(x)^2\Big)
$$
and then (\ref{U-int}) follows by (\re{inn}).
Now (\ref{wees}) with $\si=5+2\nu$ and
(\re{inn}),
(\ref{U-int}) imply
that
\beqn\nonumber
\Vert\Psi(t)\Vert^2_{L^2_{5/2+\nu}}&=&\int (1+|x|^{5+2\nu})
\Big(\int\limits_0^t \dot\Psi(x,s)ds-\Psi_0(x)\Big)^2dx\\
\nonumber
&\le& 2\int(1+|x|^{5+2\nu})\Psi^2_0(x)dx
+2t\int(1+|x|^{5+2\nu})dx\int\limits_0^t\pi^2(x,s)ds\\
\nonumber
&\le& 2d_0^2+2t\Big[\Vert X_0\Vert_{E_{5/2+\nu}}^2+U_0\Big]
\int\limits_0^t (1+s)^{6+2\nu}ds
\le C(d_0)(1+t)^{8+2\nu}
\eeqn

%%%%%%%%%%%%%%%%%%%%%%%%%%%%%%%%%%%%%%%%%%%%%%%%%%%%%%%%%%%%%%%%%%%%%%%%%%%%%%%%%%%%%%%%%%%%
\section{Proof of Proposition \ref{Amu}}
\setcounter{equation}{0}
\label{prop-pr}
%%%%%%%%%%%%%%%%%%%%%%%%%%%%%%%%%%%%%%%%%%%%%%%%%%%%%%%%%%%%%%%%%%%%%%%%%%%%%%%%%%%%%%%
First we prove the following lemma.
Let us denote by ${\cal L}(E_\si,E_{-\si})$ the Banach space of the
linear bounded operators $E_\si\to E_{-\si}$.

\begin{lemma}\label{ek}
  Let $L(\nu)$, $\nu\in\R$, be the operators $E_\si\to E_{-\si}$,
and
  \be\la{K0}
    K(t)=\int
\zeta(\nu)e^{i\nu t}Q(\nu)~d\nu,~~~~~~~Q(\nu):=\frac{L(\nu)-L(\nu_0)}
    {\nu-\nu_0}
  \ee
  where
$\zeta\in C_0^\infty(\R)$,
  and for $k=0,1,2$
  \be\la{f-dif}
 M_k:= \sup_{\nu\in\Si}
\Vert\partial ^k_\nu L(\nu)\Vert_{{\cal L}(E_\si,E_{-\si})}
<\infty
  \ee
  with $\sigma>1/2+k$ and $\Si:=\supp\zeta$. Then for $\si>5/2$
  \be\la{nK}
    \Vert K(t)\Vert_{{\cal L}(E_\si,E_{-\si})}=
{\cal O}(t^{-3/2}),\quad t\to\infty,
  \ee
\end{lemma}
\begin{proof}
Let us take $\vp\in C_0^\infty(\R)$
and
split $\zeta=\zeta_{1t}+\zeta_{2t}$,
  where
\be\la{zeze}
\zeta_{1t}(\nu):=\zeta(\nu)\vp((\nu-\nu_0)\sqrt t),~~~~
\zeta_{2t}(\nu):=\zeta(\nu)[1-\vp((\nu-\nu_0)\sqrt t)]
\ee
Then
  $$
    K(t)=\int\zeta_{1t}(\nu)e^{i\nu t}
Q(\nu)~d\nu
    +\int\zeta_{2t}(\nu)e^{i\nu t}
Q(\nu)~d\nu
    =K_1(t)+K_2(t)
  $$
Further we consider each term separately.
\\
  {\it Step i)}
For the first term we obtain
integrating twice by parts
  \beqn\la{B7}
  K_1(t)&=&\!-\frac 1{it}\!\!\!\!\!
\int\limits_{|\nu-\nu_0|<\frac 1{\sqrt t}}
  \!\!\!\!\!\!\zeta_{1t} e^{i\nu t}
  Q'(\nu)
d\nu
  -\frac 1{t^2}\!\!\!\!\!
\int\limits_{|\nu-\nu_0|
<\frac 1{\sqrt t}}\!\!\!\!\!\!
  \zeta''_{1t} e^{i\nu t}Q(\nu)d\nu
  \nonumber\\
   &-&\frac 1{t^2}\!\!\!\!\!\int\limits_{|\nu-\nu_0|
<\frac 1{\sqrt t}}\!\!\!\!\!\!
  \zeta'_{1t} e^{i\nu t}
 Q'(\nu)
d\nu
  \eeqn
For the appropriate operator norms, we have the bounds
  \beqn
  \Vert
Q(\nu)
\Vert&=&
\fr1{|\nu-\nu_0|}
\Vert
\int_{\nu_0}^\nu L'(r)dr
\Vert
\le
M_1
\nonumber\\
  \Vert Q'(\nu)\Vert
  &=&\fr1{|\nu-\nu_0|^2}
\Vert-L'(\nu)(\nu_0\!-\!\nu)\!-\!L(\nu)\!+\!L(\nu_0)\Vert
\nonumber\\
 &=&\fr1{|\nu-\nu_0|^2}
\Vert L'(\nu)\int_{\nu_0}^\nu dr-\int_{\nu_0}^\nu
L'(r)dr\Vert
\nonumber\\
 &=&
\fr1{|\nu-\nu_0|^2}
\Vert\int_{\nu_0}^\nu[L'(\nu)-L'(r)] dr\Vert
\nonumber\\
 &=&
\fr1{|\nu-\nu_0|^2}
\Vert\int_{\nu_0}^\nu[\int_r^\nu L''(s)ds] dr\Vert
\le
\fr 12 M_2
  \eeqn
 Hence, (\re{B7}) implies that
  \be\la{K1-est}
  \Vert K_1(t)\Vert_{{\cal L}(E_\si,E_{-\si})}\le C_1 t^{-3/2}.
  \ee
for $\si>5/2$ since
   $|\pa_{\nu}^k\zeta_{1t}(\nu)|\le C(k)t^{k/2}$.

  {\it Step ii)}
  For the second summand we  obtain by triple partial integration
   $$
  K_2(t)=\!-\frac 1{t^2}\int\! e^{i\nu t}\zeta_{2t}
  Q''(\nu)d\nu
  -\frac 2{t^2}\int\! e^{i\nu t}
  \zeta_{2t}'Q'(\nu)d\nu
  $$
  $$
  +\fr 1{it^3}\int e^{i\nu t}\zeta_{2t}'''Q(\nu)d\nu
  +\fr 1{it^3}\int e^{i\nu t}\zeta_{2t}'' Q'(\nu)d\nu
  $$
  $$
= K_{21}(t)+K_{22}(t)+K_{23}(t)+K_{24}(t).
  $$
Using $|\pa_{\nu}^k\zeta_{1t}(\nu)|\le C(k)t^{k/2}$, we obtain
 $$
  \Vert K_{2j}(t)\Vert_{{\cal L}(E_\si,E_{-\si})}\le C_2 t^{-3/2},\quad j=2,3,4.
$$
Finally, to estimate $K_{21}(t)$, we use the identity
  \beqn
  Q''(\nu)
  &=&\fr{L''(\nu)(\nu-\nu_0)^2-2(L(\nu_0)-L(\nu)-L'(\nu)(\nu_0-\nu))}{(\nu-\nu_0)^3}
\nonumber\\
&=&\fr{\ds L''(\nu)(\nu-\nu_0)^2-2
\int_{\nu_0}^\nu[\int_r^\nu L''(s)ds] dr
}{(\nu-\nu_0)^3}
  \eeqn
which implies that
$\Vert Q''(\nu)\Vert\le C M_2/|\nu-\nu_0|$. Therefore,
 $$
  \Vert K_{21}(t)\Vert_{{\cal L}(E_\si,E_{-\si})}\le C t^{-3/2}.
 $$
 since $\zeta_{2t}(\nu)=0$ for $|\nu-\nu_0|\le \fr1{2\sqrt{t}}$.

\end{proof}

%%%%%%%%%%%%%%%%%%%%%%%%%%%%%%%%%%%%%%%%%%%%%%%%%%%%%%%%%%%%%%%%%%%%%%%%%%%%%%%%
\noindent
{\bf Proof of Proposition \ref{Amu}}
  The operator $e^{At}(A-2i\mu-0)^{-1}$ admits the Laplace
  representation
  $$
    e^{At}(A-2i\mu-0)^{-1}=-\frac 1{2\pi i}\int\limits_{-i\infty}^{i\infty}
    e^{\lam t}R(\lam+0)d\lam ~ R(2i\mu+0).
  $$
  Let us apply the Hilbert identity for the resolvent:
  $$
    R(\lam_1)R(\lam_2)=\frac 1{\lam_1-\lam_2}[R(\lam_1)-R(\lam_2)],\;\Re\lam_1,
    \Re\lam_2>0,
  $$
  for $\lam_1=\lam+0$ and $\lam_2=2i\mu+0$. Then we obtain
  $$
    e^{At}(A-2i\mu-0)^{-1}=-\frac 1{2\pi i}\int\limits_{-i\infty}^{i\infty}
    e^{\lam t}\frac{R(\lam+0)- R(2i\mu+0)}{\lam-2i\mu}~d\lam
  $$
  $$
    =-\frac 1{2\pi i}\!\int\limits_{-i\infty}^{i\infty}\!
    e^{\lam t}\zeta(\lam)\frac{R(\lam+0)- R(2i\mu+0)}{\lam-2i\mu}~d\lam
    -\frac 1{2\pi i}\!\!\int\limits_{{\cal C}_+\cup{\cal C}_-}\!\!
    e^{\lam t}(1-\zeta(\lam))\frac{R(\lam+0)- R(2i\mu+0)}{\lam-2i\mu}~d\lam
  $$
  $$
    -\frac 1{2\pi i}\int\limits_{(-i\infty,i\infty)\setminus({\cal C}_+\cup{\cal C}_-)}
    e^{\lam t}(1-\zeta(\lam))\frac{R(\lam+0)- R(2i\mu+0)}{\lam-2i\mu}~d\lam
    =K_1(t)+K_2(t)+K_3(t),
  $$
  where $\zeta(\lam)\in C_0^{\infty}(i\R)$,
  $\zeta(\lam)=1$ for $|\lam-2i\mu|<\delta/2$ and $\zeta(\lam)=0$ for $|\lam-2i\mu|>\delta$,
  with $0<\delta<2\mu-\sqrt{2}$.
  By Lemma \ref{ek} with $L(\nu)=R(i\nu+0)$,
  $$
    \Vert K_1(t)\Vert_{{\cal L}(E_\si,E_{-\si})}={\cal O}({t^{-3/2}}),
    \quad t\to\infty
  $$
  since $\si>5/2$.
The bounds (\ref{f-dif}) for $L(\nu)$ follow from
  Proposition \re{1d}.
  For the operator $K_2(t)$ we also apply Proposition \re{1d}
 and obtain
  $$
    \Vert K_2(t)\Vert_{{\cal L}(E_\si,E_{-\si})}={\cal O}({t^{-3/2}}),\quad t\to\infty
  $$
Here the choice of the sigh in $A-2i\mu-0$ plays the crucial role.
  Further, the integrand in $K_3(t)$ is an analytic function of
  $\lam\not =0, \pm i\mu$ with the values in
${\cal L}(E_\si,E_{-\si})$
for $\beta\ge 0$. At the points $\lam=0$ and $\lam=\pm i\mu$ the
  integrand has the poles of finite order. However, all the Laurent
  coefficients vanish when applied to $P_{c}h$. Hence for $K_3(t)$
  we obtain, twice integrating by parts,
  $$
    \Vert K_3(t) P^c h\Vert_{E_{-\si}}
    \le c(1+t)^{-2}\Vert h\Vert_{E_{\si}},
  $$
completing the proof.

%%%%%%%%%%%%%%%%%%%%%%%%%%%%%%%%%%%%%%%%%%%%%%%%%%%%%%%%%%%%%%%%%%%%%%%%%%%%%%%%%%%%%%
%%%%%%%%%%%%%%%%%%%%%%%%%%%%%%%%%%%%%%%%%%%%%%%%%%%%%%%%%%%%%%%%%%%%%%%%%%%%%%%%%%%%%%
\section{Examples}
 \label{examples}
%%%%%%%%%%%%%%%%%%%%%%%%%%%%%%%%%%%%%%%%%%%%%%%%%%%%%%%%%%%%%%%%%%%%%%%%%%%%%%%%%%%%%%
We construct the examples of the potentials $U(\psi)$
satisfying all the conditions {\bf U1} -- {\bf U3}.
We will construct $U(\psi)$ by small perturbation of the
cubic  Ginzburg-Landau potential
\be\la{GL}
  U_{0}(\psi):=\fr 14(1-\psi^2)^2
\ee
For the potential $U_{0}(\psi)$ the kink is explicitly given by
\be\la{GLkink}
 s_{0}(x):=\tanh \fr{x}{\sqrt 2}
\ee
The potential $V_0$ of the linearized eduation reads
$$
V_{0}(x)=U_{0}''(s_{0}(x))-2=-3\cosh^{-2}\fr{x}{\sqrt 2}
$$
Let us consider the corresponding  Schr\"odinger operator
$$
H_{0}=-\frac{d^2}{dx^2}+2+ V_0(x)
=-\frac{d^2}{dx^2}+2-\frac{3}{\cosh^2(x/\sqrt 2)}
$$
restricted to the subspace of odd functions (\re{odd}).
The continuous spectrum of the operator $H_0$ coincides with the interval
$[2,\infty)$.
It is well known  (see \cite{GKKG}, pp. 64-65) that
\medskip\\
{\bf i)}
The discrete spectrum of $H_0$ consists  only of one point $\lam_0=3/2$
with the corresponding eigenfunction
$\varphi_0=\sinh(x/\sqrt 2)\big/\cosh^2(x/\sqrt 2)$;
\medskip\\
{\bf ii)} The end point $\lam=2$ of the continuous spectrum is not
eigenvalue nor resonance.
\medskip\\
Hence,  the Condition {\bf U2} holds for the potential $U_0$.
Further, the non-degeneracy condition {\bf U3} reads \be\la{FGR1}
\int \phi_6(x)\frac{\sinh^3(x/\sqrt 2)}{\cosh^5(x/\sqrt 2)}dx\not=0
\ee where $\phi_6(x)$ is a nonzero odd solution to
$$
H_0\phi_6(x)=6\psi_6(x)
$$
Numerical calculation \cite{ks} demonstrate the validity of the
condition (\ref{FGR1}) and hence  {\bf U3} holds.

The potential  $U_0(\psi)$ satisfies the conditions (\ref{U1})
with $a=1$ and $m^2=2$.
However,  $U_0(\psi)$ does not satisfy the conditions (\ref{U11})
since $U_0'''(\pm 1)=\pm 6$, $U_0^{(4)}(\pm 1)=6$.

Therefore we will construct a small perturbation of the potential $U_{0}$.
Namely, for an appropriate  fixed $C>0$, and any sufficiently small $\de>0$,
there exist the potentials $U(\psi)$ satisfying (\ref{U11}) such that
\be\la{U2}
  \left.\begin{array}{l}
  U(\psi)=U_0(\psi)~~{\rm for}~~ ||\psi|- 1|>\de,\quad
  \sup\limits_{\psi\in\R}|U^{(k)}(\psi)-U_{0}^{(k)}(\psi)|\le C\de,
  \quad k=0,1,2,
   \\
  \sup\limits_{\psi\in\R}|U'''(\psi)-U_{0}'''(\psi)|\le C
  \end{array}\right|
\ee
For example, let us set
$$
U(\psi)=U_{0}(\psi)-\Big[\fr 14(|\psi|-1)^4
+(|\psi|-1)^3\Big]\chi_\de(|\psi|-1)
$$
where  $\chi_\de(z)=\chi(z/\de)$, $\chi(z)\in C_0^\infty(\R)$,
$\chi(z)=1$ for $|z|<1/2$, and  $\chi(z)=0$ for $|z|>1$.
Then the conditions (\ref{U2}) holds, and
$$
U(\psi)=(|\psi|- 1)^2~~ {\rm for}~~ ||\psi|- 1|<\de/2,~~ {\rm and}~~
U(\psi)=U_{0}(\psi)~~ {\rm for}~~||\psi|- 1|>\de
$$
Hence, $U(\psi)$ satisfies {\bf U1}.
It remains to prove that  $U(\psi)$ satisfies {\bf U2} and {\bf U3}.

Denote ${\cal S}=\{x\in\R:||s(x)|- 1|,~||s_{0}(x)|- 1|<\de\}$.
Then $s(x)=s_{0}(x)$ and $V(x)=V_0(x)$ for $x\in\R\setminus\cal S$.
For $x\in\cal S$, using (\ref{U2}), we obtain
\beqn\nonumber
\sup\limits_{x\in\cal S}|V(x)-V_{0}(x)|&\le&
\sup\limits_{x\in\cal S}|U''(s(x))- U''(s_0(x))|+
\sup\limits_{x\in\cal S}|U''(s_{0}(x))- U_0''(s_{0}(x))|\\
\nonumber
&=&\sup\limits_{||\phi|- 1|<\de} |U'''(\phi)|
\sup\limits_{x\in\cal S}|s_{0}(x)- s(x)|+\cO(\de)=\cO(\de)
\eeqn
since $\sup\limits_{x\in\cal S}|s_{0}(x)- s(x)|\le 2\de$. Hence
\be\la{V-diff}
\sup\limits_{x\in\R}|V(x)-V_{0}(x)|=\cO(\de)
\ee
Let us verify the uniform decay of $V(x)$ for small $\de>0$.
We consider the case $x\ge 0$ (the case $x\le 0$ can be consider
similarly). Note that $U(\psi)\ge (\psi-1)^2/4$ for $0\le \psi<1$.
Using the identity
$$
\int\limits_0^{s(x)}\frac{ds}{\sqrt{2U(s)}}=x
$$
we obtain for $x>0$ and $0\le s(x)<1$
$$
x\le\int\limits_0^{s(x)}\frac{\sqrt 2~ds}{\sqrt{(1-s)^2}}
=\int\limits_0^{s(x)}\frac{\sqrt 2~ds}{1-s}=-\sqrt 2~\ln(1-s(x))
$$
Hence, $1-s(x)\le e^{-x/\sqrt 2}$ for $x\ge 0$, and then
$$
|1-|s(x)||\le e^{-|x|/\sqrt 2},\quad x\in\R
$$
Therefore
\be\la{uni}
|V(x)|\le Ce^{-|x|/\sqrt 2},\quad x\in\R
\ee
Finally, the unifirm bounds (\ref {V-diff}) and (\ref{uni}) imply
that the conditions {\bf U2} and {\bf U3} hold for the potentials
$U(\psi)$ for sufficiently small $\de>0$ since they hold for $U_0(\psi)$.

%%%%%%%%%%%%%%%%%%%%%%%%%%%%%%%%%%%%%%%%%%%%%%%%%%%%%%%%%%%%%%%%%%%%%%%%%%%%%

\end{document}